\documentclass[a4paper,12pt]{article}

\usepackage[english]{babel}
\usepackage[utf8]{inputenc}
\usepackage{etex}
\usepackage{pdfpages}% pour inclure des pages pdf avec \includepdf[pages=2-8]{file.pdf}
\usepackage{amsfonts}
\usepackage{amsmath}
\usepackage{amssymb}
\usepackage{amsthm}
\usepackage{appendix}
\usepackage{hyperref}
\hypersetup{colorlinks,linkcolor={blue},citecolor={blue},urlcolor={red}} 
\usepackage{cases}
\usepackage{enumerate}
\usepackage{caption}
\usepackage{setspace}
\usepackage{ulem}%pour hachrer du texte
\usepackage{enumitem}
\usepackage{scalerel}

\usepackage{upgreek}
\usepackage{fancyhdr}

\usepackage{manfnt} %lastpage
\usepackage{mathrsfs}
\usepackage{minitoc}
\usepackage{natbib}
\usepackage{accents}

\newcommand{\q}[1]{``#1''}
\usepackage{amsmath,amsfonts,bbm,amsthm}
\usepackage{cleveref}       % smart cross-referencing
\usepackage{amssymb}
\usepackage{float}
\usepackage{latexsym}
\usepackage{eucal}
\usepackage{mathtools}
\usepackage{tikz}
\usetikzlibrary{decorations.markings}
\usepackage{pgfplots}
\usepackage{lipsum}
\usepackage{color}
\usepackage{tabularx}
\usepackage{array}
\usepackage[linesnumbered,ruled,vlined]{algorithm2e}
\SetKwInput{KwInput}{Input}
\SetKwInput{KwOutput}{Output}
\SetKwInput{KwEnsure}{Ensure}
\usepackage{graphicx,caption,subcaption}
\usepackage{moreverb}
\usepackage{multirow} %% pour regrouper un texte sur plusieurs lignes dans une table
\usepackage{url} %% pour citer les url par \url
\usepackage[all]{xy} %% pour la barre au dessus des symboles
\usepackage{shorttoc} %% pour plusieurs tables des matières par la commande \shorttableofcontents{Titre}{profondeur}.
\usepackage{textcomp} %% pour le symbol pour mille par \textperthousand.
\usepackage{hyperref} %% pour la transformation en PDF, ça permet d'obtenir des liens sur les sections ...
\usepackage{url}
\usepackage{enumerate}%\usepackage{textcomp}
\usepackage{enumitem}
\usepackage[right]{eurosym}
\def \R{\mathbb{R}}

\def \N{\mathbb{N}}

\newcommand{\argmin}{\mathrm{argmin}}
\newcommand{\argmax}{\mathrm{argmax}}

\usepackage{graphics}
\usepackage{graphicx}
\usepackage{hyperref}

\usepackage{lastpage}
\usepackage{latexsym}

\usepackage{color}
\usepackage{comment}

\usepackage{enumerate}

\usepackage{fancyhdr}
\usepackage[top=3cm, bottom=3cm, left=1.5cm , right=2cm]{geometry}%Ajustement des marges
\newtheorem{Theorem}{Theorem}[section]
\newtheorem{Corollary}[Theorem]{Corollary}

\newtheorem{Definition}[Theorem]{Definition}
\newtheorem{Proposition}[Theorem]{Proposition}
\newtheorem{Remark}[Theorem]{Remark}

\newtheorem{Example}{Example}

\numberwithin{equation}{section}%permet de numeroter les equations suivant le paragraphe auquel elles appartiennent

\usepackage[framemethod=tikz]{mdframed}
\usetikzlibrary{decorations.pathmorphing,calc}
\newcommand\wavydecor{%
    \draw[decoration={coil,aspect=0.1,segment length=5pt,amplitude=1.0pt},decorate,line width=1.5pt,black]
      (O|-P) -- (O);
}
\newmdenv[
hidealllines=true,
innerleftmargin=10pt,
innerrightmargin=0pt,
innertopmargin=0pt,
innerbottommargin=0pt,
leftmargin=-10pt,
skipabove=.5\baselineskip,
skipbelow=.5\baselineskip,
singleextra={\wavydecor},
firstextra={\wavydecor},
secondextra={\wavydecor},
middleextra={\wavydecor}
]{done}
\usepackage[linecolor=black,textwidth=25mm,textsize= footnotesize]{todonotes} %to help annotate the working paper%

{ \begin{list}%
	{$\bullet$}%
	{\setlength{\labelwidth}{10pt}%
	 \setlength{\leftmargin}{35pt}%
	 \setlength{\itemsep}{\parsep}{1pt}}}%
{ \end{list} }

\title{ Multivariate Optimized Certainty Equivalent Risk Measures and their Numerical Computation\thanks{Acknowledgements:  The authors research is part of the ANR project DREAMeS (ANR-21-CE46-0002) and benefited from the support of the "Chair Risques Emergents en Assurance" under the aegis of Fondation du Risque, a joint initiative by Le Mans University and Cov\'ea.}}  

\author{Sarah Kaaka\"i~ \thanks
{\small Laboratoire Manceau de Math\'ematiques \& FR CNRS N\textsuperscript{o} 2962, Institut du Risque et de l'Assurance, Le Mans University.}
\and Anis Matoussi
\thanks{ \small  Laboratoire Manceau de Math\'ematiques \& FR CNRS N\textsuperscript{o} 2962, Institut du Risque et de l'Assurance, Le Mans University and Ecole Polytechnique.
}\\
 \and Achraf Tamtalini~ \thanks
{\small  Laboratoire Manceau de Math\'ematiques \& FR CNRS N\textsuperscript{o} 2962, Institut du Risque et de l'Assurance, Le Mans University.}
}
\date{}

\begin{document}
 \maketitle

\abstract{   We present a framework for constructing multivariate risk measures that is inspired from univariate Optimized Certainty Equivalent (OCE) risk measures. We show that this new class of risk measures verifies the desirable properties such as convexity, monotonocity and cash invariance. We also address numerical aspects of their computations using stochastic algorithms instead of using Monte Carlo or Fourier methods that do not provide any error of the estimation.
  }

\bigskip

\noindent  {\bf Keywords}: Multivariate risk measures, Optimized certainty equivalent, Numerical methods, stochastic algorithms, risk allocations.

%=================================================================
\section*{Introduction}
One of the major concerns in finance is how to assess or quantify the risk associated with a random cashflow in the future. Starting with the pioneering work of \citet{markowitz1952}, the risk associated with a random outcome was quantified by its variance.
Then, \citet{artzner} published their famous seminal paper in which they introduce the theory of risk measures. In their paper, risk measures were defined as a map verifying certain properties, which are called \q{axioms}, namely: Subadditivity, translation invariance, monotonocity and positive homogeneity. Such risk measures are called \textit{coherent} risk measures. Many extensions have been proposed and studied in the literature after the introduction of the axiomatic approach. One important extension is the notion of \textit{convex risk measure} developed by \citet{follmer} and \citet{frittelli2002putting} where the subadditivity and positive homogeneity properties were replaced by the weaker property of convexity. The latter reflects the fact that diversification decreases the risk. In the banking industry, one of the most popular risk measures is the Value at Risk (VaR in short). This is due first, to its financial interpretation and second, to its easy and fast implementation. Indeed, VaR is defined as the minimal cash amount that needed to be added to a financial position in order to have a probability of losses below a certain threshold. Its computation amounts to the calculation of a quantile of the portfolio distribution. Nevertheless, VaR suffers from one drawback: it does not verify the convexity property. This has prompted the search for new examples of risk measures, the most prominent being the Conditional Value at Risk (CVaR), the entropic risk measure and the utility based risk measure (also known as shortfall risk measure).\\
Some decision making problem based on utility functions are closely related to risk measures. One can cite the optimized certainty equivalent (OCE) that was first introduced by \citet{ben1986expected}. The idea behind the definition of OCE is as follows: Assume that a decision maker, with some utility function $u$, is expecting a random income $X$ in the future and can consume a part of it at present. If he chooses to consume $m$ dollars, the resulting present value of $X$ is then $P(X, m) \vcentcolon = m + E[u(X-m)]$. Hence, one can define the sure present value of $X$ (i.e., its certainty equivalent) as the result of an optimal allocation of $X$ between present and future consumption, that is the decision maker will try to find $m$ that maximizes $P(X,m)$. The main properties of the OCE were studied in \citet{ben2007old} where it is showed that the opposite of the OCE provides a wide family of risk measures that verifies the axiomatic formalism of convex risk measures. They also proved that several risk measures, such as CVaR and the entropic risk measure, can be derived as special cases of the OCE by using particular utility functions (see also \citet{cherny2007divergence}).\\
From a systemic point of view, the financial crisis of $2008$ has demonstrated the need for novel approaches that capture the risk of a system of financial institutions. More precisely, given a network/system of $d \in \N$ different but dependent portfolios $X \vcentcolon =(X_1,..., X_d)$, we are interested in measuring/quantifying the risk carried by this system of portfolios. A classical approach consists in first aggregating the portfolios using some aggregation function $\Lambda: \R^d \to \R$ and then apply some univariate risk measure applied to the aggregated portfolio. In practice, most of the times the aggregation function is just the sum of the components, i.e., $\Lambda(x) = \sum_{i=1}^d x_i$. This will result in having a systemic risk measure of the form: $R(X) = \eta(\Lambda(X)) = \eta(\sum_{i=1}^d X_i)$, where $\eta$ is a univariate risk measure, such as the VaR, CVaR, entropic risk measure, etc. The mechanism behind this approach is also known as \q{Aggregate then Inject Cash} mechanism (see \citet{unified}). However, this approach suffers from one major drawback: While it quantifies the systemic risk carried by the whole system, it does not provide risk levels of each portfolio, and thus, one could not have a ranking of portfolios in terms of their systemic riskiness. One way to remediate to this, is to consider the reverse mechanism, that is to \q{Inject Cash then Aggregate}. This consists in associating to each portfolio a risk measure and summing up the resulting risk levels. This results in considering systemic risk measures $R(X)$ of the following form: $R(X) = \sum_{i=1}^d \eta_i(X_i)$, where $\eta_i$'s are the univariate risk measures associated to each portfolio. Obviously, one could use the same univariate for all portfolios, that is $\eta_i = \eta, \forall i \in \{1,...,d\}$. However, by doing so, we are assuming that the system is made of \q{isolated} portfolios with no interdependence structure, and hence, we might be overestimating or underestimating the systemic risk. This led several authors to look for approaches that address simultaneously the design of an overall risk measure and the allocation of this risk measure among the different components of the system. In this spirit, an extension of shortfall risk measures, introduced in \citet{follmer2002convex}, has been studied in \citet{armenti} based on multivariate loss functions. However, one should note that, to ensure the existence of optimal allocation problem, these loss functions must verify a key property: permutation invariance. In other words, each component of the system is treated as if it has the same risk profile as all the other components and thus one cannot discriminate a particular component against one another. Moreover, classical risk measures such that the CVaR and the entropic risk measure cannot be recovered using multivariate shortfall risk measures, which limit their use in practice. We will see that, with our multivariate extension of OCE risk measure, the permutation invariance condition is no longer needed and by choosing the appropriate loss functions, we can retrieve most of the classical risk measures.\\
One of the major issues that arises when studying risk measures is their numerical approximation. The standard VaR can be computed by inverting the simulated empirical distribution of the financial position using Monte Carlo (see \citet{glasserman} and \citet{glasserman1}). An alternative method for computing VaR and CVaR is to use stochastic algorithms (SA). The rational idea behind this perspective comes from the fact that both VaR and CVaR are the solutions and the value of the same convex optimization problem as pointed out in \citet{rockafellar2002conditional} and the fact that the objective function is expressed as an expectation. This was done in \citet{pages}, where they prove the consistency and the asymptotic normality of the estimators. In the same direction, in \citet{hal2}, we extended the work of \citet{stochastic}  to approximate multivariate shortfall risk measures using stochastic algorithms. In \citet{neufeld2008antonis}, they developed numerical schemes for the computations of univariate OCE using Fourier transform methods.\\
The outline of this paper is as follows: in section \ref{MOCE}, we give the definition of multivariate OCE by introducing first the class of appropriate loss functions. Then, we show that this class of risk measures verifies the desirable properties. We also characterize the optimal solutions, give a dual representation and study the sensitivity with respect to external shocks. Finally, section \ref{ComputationalAspects} treats the computational aspects of approximating multivariate OCE using a deterministic scheme and a stochastic one. 
\section{Multivariate OCE}\label{MOCE}
Let $(\Omega, \mathcal{F}, P)$ a probability space and we denote by $L^0(\R^d)$ the space of $\mathcal{F}$- measurable random vectors taking values in $\R^d$. For $x, y$ in $\R^d$, we denote by $||\cdot||$ the Euclidean norm and $x \cdot y = \sum x_i y_i$. For a function $f :\R^d \to [-\infty, \infty]$, we define $f^*$ the convex conjugate of $f$ as $f^*(y) = \sup_x \{x \cdot y - f(x)\}$. The space $L^0(\R^d)$ inherits the lattice structure of $\R^d$ and hence, we can use the classical notations in $\R^d$ in a $P$-almost-surely sens. We will say for example, for $X, Y \in L^0(\R^d)$ that $X \ge Y$ if $P(X \ge Y) =  1$. To alleviate the notations, we will drop the reference to $\R^d$ in $L^0(\R^d)$ whenever it is unnecessary. For $Q = (Q_1,...,Q_d)$ a vector of probabilities, we will write $Q \ll P$ if for all $i=1,...,d$, we have $Q_i \ll P$.
In this section, we introduce the notion of multivariate Optimized Certainty Equivalent (OCE) and give its main properties. The latter is an extension of univariate OCE that was introduced and studied in details in \citet{ben2007old}. First, we start by giving the definition of a multivariate loss function that will be used in the rest of the paper. For the rest of the paper, the random vector $X = (X_1,...,X_d) \in L^0$ represents profits and losses of $d$ portfolios.
\begin{Definition}\label{defLoss}
A function $l:\R^d \mapsto (-\infty, \infty]$  is called a loss function, if it satisfies the following properties:
\begin{enumerate}
\item $l$ is nondecreasing, that is if $x \le y$ componentwise, then $l(x) \le l(y)$.
\item $l$ is lower-semicontinuous and convex.
\item $l(0) = 0$ and  $l(x) > \sum_{i=1}^d x_i,~\forall x \neq 0$.
\end{enumerate}
\end{Definition}
\noindent For integrability reasons, we will work in the multivariate Orlicz heart defined as:
$$ M^\theta \vcentcolon = \{ X \in L^0: E[\theta(\lambda X)] < \infty, \forall \lambda > 0\},$$
where $\theta(x) = l(|x|), x \in \R^d$. On this space, we define the Luxembourg norm as:
$$ ||X||_\theta \vcentcolon = \left\{ \lambda > 0, E\left[ \theta \left(\frac{|X|}{\lambda}\right)\right] \le 1 \right\}.$$
Under the Luxembourg norm, $M^\theta$ is a Banach lattice and its dual with respect to this norm is given by the Orlicz space $L^{\theta^*}$:
$$L^{\theta^*} \vcentcolon = \{ X \in L^0, E[\theta^*(\lambda X)] < \infty,~\text{for some}~ \lambda > 0\}.$$ 
We also introduce the set of $d$-dimensional measure densities in $L^{\theta^*}$, that is:
$$ \mathcal{Q}^{\theta^*} \vcentcolon = \left\{\frac{dQ}{dP} \vcentcolon = (Z_1,...,Z_d), Z \in L^{\theta^*}, Z_k \ge 0~\text{and}~E[Z_k] = 1\right\}.$$
Note that for $Q \in \mathcal{Q}^{\theta^*}$ and $X \in M^\theta$, $\frac{dQ}{dP}\cdot X \in L^1$, thanks to Fenchel inequality and for the sake of simplicity, we will write $E_Q[X] \vcentcolon= E[dQ/dP \cdot X]$. We refer to Appendix B in \citet{armenti} for more details about multivariate Orlicz spaces.
\begin{Definition}
Assume $l$ is a loss function. The multivariate OCE risk measure is defined for every $X \in M^\theta$ as:
\begin{equation}\label{OCE}
R(X) = \underset{w \in \R^d}{\inf}\left\{\sum_{i=1}^d w_i + E[l(-X - w)]\right\}.
\end{equation}
\end{Definition}
\begin{Example}
\label{example}
When $d=1$, we can recover some important convex risk measures such CVaR (also called Expected Shortfall or Average Value at Risk) and Entropic risk measure.
\begin{enumerate}
\item CVaR: Let $\alpha\in (0,1)$ and take $l(x)=\frac{1}{1-\alpha}x^+$, then the associated risk measure is the CVaR (see \citet{rockafellar2002conditional}).
\item Polynomial loss function: For an integer $\gamma > 1$, the polynomial loss function is defined by: $l(x)=\frac{([1+x]^+)^\gamma-1}{\gamma}$. When $\gamma = 2$, the corresponding risk measure is the Monotone Mean-Variance (see \citet{vcerny2012computation}).
\item Entropic risk measure: Fix $\lambda > 0$ and let $l(x) \vcentcolon=\frac{\exp(\lambda x)-1}{\lambda} $. Then, the problem in \eqref{OCE} can be explicitly solved and the optimal $w^*$ and $R(X)$ are given by: 
$$w^* = \frac{1}{\lambda}\log(E[e^{-\lambda X}]),~R(X) = w^* =  \frac{1}{\lambda}\log(E[e^{-\lambda X}]).$$
\end{enumerate} 
Using univariate loss functions, we can construct multivariate loss functions in the following way: Given $l_1,..., l_d$ univariate loss functions and a nonnegative, convex and lower-semicontinuous function $\Lambda:\R \to \R^+$ with $\Lambda(0)=0$,
one can define a multivariate loss function as follows:
\begin{equation}\label{constructionLoss}
l(x) \vcentcolon= \sum_{i=1}^d l_i(x_i) + \Lambda(x).
\end{equation}
It is easy to see that $l$ verifies all the conditions in the definition \ref{defLoss}. Note that by taking $\Lambda$ the null function, the corresponding multivariate OCE boils down to a sum of univariate OCE. It is in this function $\Lambda$ where the dependence between the different components in the system is taken into account. In this paper, we will focus on the following multivariate loss functions inspired from the univariate risk measures above:
\begin{align}
l(x) &= \sum_{i=1}^d \frac{e^{\lambda_i x_i}-1}{\lambda_i} + \alpha e^{\sum_{i=1}^d \lambda_i x_i}, ~\lambda_i > 0,~\alpha \ge 0,\label{exponential}\\
l(x) &= \sum_{i=1}^d \frac{([1 +x_i]^+)^{\theta_i} - 1}{\theta_i} + \alpha \sum_{i < j} \frac{([1 +x_i]^+)^{\theta_i}}{\theta_i} \frac{([1 +x_i]^+)^{\theta_j}}{\theta_j},~ \theta_i > 1,~\alpha \ge 0,\label{polynomial}\\
l(x) &= \sum_{i=1}^d \frac{x_i^+}{1 -\beta_i} + \alpha \sum_{i < j} \frac{x_i^+}{1 -\beta_i} \frac{x_j^+}{1 -\beta_j},~0 < \beta_i < 1,~\alpha \ge 0.\label{cvar}
\end{align}
\end{Example}
In the next theorem, we show that the multivariate OCE is a convex risk measure as defined in \citet{follmer2002convex}. 
\begin{Theorem}
The function $R$ in \eqref{OCE} is real valued, convex, monotone and cash invariant\footnote{In the following sens: $R(X+m)=R(X) - \sum_{i=1}^d m_i$} risk measure. In particular, it is continuous and subdifferentiable. If $l$ is positive homogeneous, then $R$ is too. Furthermore, it admits the following representation:
\begin{equation}\label{Representation}
R(X) = \underset{Q \in \mathcal{D}^{l^*}}{\max} \{E_Q[-X] - \alpha(Q)\},
\end{equation}
where the penalty function $\alpha$ is defined for $Q = (Q_1,...,Q_d) \ll P$ by: $\alpha(Q) = E\left[l^*\left(\frac{d Q}{d P}\right)\right]$ and $\mathcal{D}^{l^*}=  \{Q \ll P, ~\alpha(Q) < \infty \} \vcentcolon = dom(\alpha)$.
\end{Theorem}
\begin{proof}\
\begin{itemize}
\item $R(X) \in \R$ for all $X \in M^\theta$: Since $M^\theta \subseteq L^1$, by the third property of loss functions, we have for every $X\in M^\theta$ and $w \in \R^d$: $\sum_{i=1}^d w_i + E[l(-X - w)] \ge E[-\sum_{i=1}^d X_i] > -\infty$. $R(X) < + \infty$ since for $w=0$, we have $E[l(-X)] < \infty$.
\item Monotonicity: Let $X, Y \in M^\theta$ such that $X \le Y$. Since $l$ is non-decreasing, then $E[l(-X - w)] \ge E[l(-Y - w)]$ for every $w \in \R^d$, which in turn implies $R(X) \ge R(Y)$. 
\item Convexity: Let $X, Y \in M^\theta$ and $\lambda \in [0, 1]$. We have thanks to the convexity of $l$:
\begin{align*}
R(\lambda X + (1 -\lambda)Y) &= \underset{w \in \R^d}{\inf}\{\sum_{i=1}^d w_i + E[l(-X-w)]\}\\
&= \underset{w^1, w^2 \in \R^d}{\inf}\left\{\sum_{i=1}^d \lambda w^1_i + (1-\lambda)w^2_i + E[l(\lambda (-X - w^1) + (1-\lambda)(-Y - w^2))]\right\}\\
&\le \underset{w^1 \in \R^d}{\inf}\underset{w^2 \in \R^d}{\inf}\left\{\lambda (\sum_{i=1}^d w^1_i + E[l((-X - w^1)] + (1-\lambda)(\sum_{i=1}^d w^2_i + E[l((-X - w^2)])\right\}\\
&= \lambda R(X) + (1 -\lambda)R(Y).
\end{align*}
\item Cash Invariance: Let $m\in \R^d$, we have:
\begin{align*}
R(X+m)&=\underset{w \in \R^d}{\inf}\left\{\sum_{i=1}^d w_i +E[l(-X-m-w)]\right\}\\
&= \underset{w \in \R^d}{\inf}\left\{\sum_{i=1}^d (w_i + m_i) - \sum_{i=1}^d m_i + E[l(-X-m-w)]\right\}\\
&=R(X) - \sum_{i=1}^d m_i
\end{align*}
\item Continuity and subdifferentiability: Since $(M^\theta, ||\cdot||_\theta)$ is a Banach space, this is a direct consequence of Theorem 4.1 in \citet{cheridito2009risk} or Theorem 1 in \citet{biagini2009extension}.
\item Positive homogeneity: If $l$ is positive homogeneous, then by the definition of $R(X)$, we have for $\lambda > 0$:
\begin{align*}
R(\lambda X) &= \underset{w \in \R^d}{\inf}\left\{\sum_{i=1}^d w_i +E[l(-\lambda X - w)]\right\}\\
&= \underset{w \in \R^d}{\inf}\left\{\sum_{i=1}^d w_i +\lambda E[l(-X-\frac{w}{\lambda})]\right\}\\
&= \lambda \underset{w \in \R^d}{\inf}\left\{\sum_{i=1}^d \frac{w_i}{\lambda} +E[l(-X-\frac{w}{\lambda})]\right\}\\
&= \lambda \underset{w \in \R^d}{\inf}\left\{\sum_{i=1}^d w_i +E[l(-X-w)]\right\} = \lambda R(X).
\end{align*}
\item Representation: First, because $R$ is convex and continuous, Fenchel-Moreau theorem implies that:
\begin{equation}\label{representation}
R(X) = \underset{Y \in L^{\theta^*}}{\sup}\{E[X\cdot Y] - R^*(Y)\} = \underset{Y \in L^{\theta^*}}{\max}\{E[X\cdot Y] - R^*(Y)\}
\end{equation}
where $R^*(Y) = \sup\{E[X\cdot Y] - R(X), X \in M^\theta\}, Y \in L^{\theta^*}$. Now, if $Y \nleq 0$, then by the bipolar theorem, there exists $X_1 \in M^\theta$ such that $X_1  \ge 0$ and $E[X_1 \cdot Y] > 0$. Using the definition of $R^*(Y)$, we get the following:
\begin{align*}
R^*(Y) &= \underset{X \in M^\theta}{\sup}\{E[X\cdot Y]- R(X)\}\\
&\ge \underset{\lambda > 0}{\sup} \{\lambda E[X_1 Y] - R(\lambda X_1)\}\\
&\ge \underset{\lambda > 0}{\sup} \{\lambda E[X_1 Y] \}- R(0)\} = +\infty,
\end{align*}
where the last inequality is due to the monotonicity of $R$. Therefore, the maximum can be taken over $Y \le 0$. For $k \in \{1,...,d\}$, let $X = (0,...,x,...)$ and $x > 0$. By the translation invariance property, we have $R(X_k) = R(0) - x$. Consequently,
\begin{align*}
R^*(Y) &= \underset{X \in M^\theta}{\sup}\{E[X\cdot Y]- R(X)\}\\
&\ge x E[Y_k] - R(0) + x = x(E[Y_k]+1)-R(0).
\end{align*}
If $E[Y_k] \neq -1$, then by sending $x$ to infinity, we get that $R^*(Y) = \infty$. Finally, this shows that the maximum in \eqref{representation} could be taken over $\mathcal{D}^{\theta^*}$, i.e., $R(X) = \underset{Q \in \mathcal{D}^{\theta^*}}{\max}\{E[-\frac{dQ}{dP}\cdot X] - R^*(-Q)\} = \underset{Q \in \mathcal{D}^{\theta^*}}{\max}\{E_Q[-X] - R^*(-Q)\}$. Let us now explicit more the expression of $R^*(-Q)$ for $Q \in \mathcal{D}^{\theta^*}$:
\begin{align*}
R^*(-Q) &= \underset{X \in M^\theta}{\sup} \left\{E\left[-\frac{dQ}{dP} \cdot X\right] - R(X)\right\}\\
&= \underset{X \in M^\theta}{\sup}\left\{E\left[-\frac{dQ}{dP} \cdot X\right] -\left( \underset{m \in \R^d}{\inf}\sum_{i=1}^d m_i + E[l(-X-m)]\right)\right\}\\
&= \underset{X \in M^\theta}{\sup}\left\{E\left[-\frac{dQ}{dP} \cdot X\right]+ \underset{m \in \R^d}{\sup}\left(-\sum_{i=1}^d m_i - E[l(-X-m)]\right)\right\}\\
&= \underset{X \in M^\theta}{\sup}\underset{m \in \R^d}{\sup}\left\{E\left[-\frac{dQ}{dP} \cdot X\right]-\sum_{i=1}^d m_i - E[l(-X-m)]\right\}\\
&=\underset{m \in \R^d}{\sup}\underset{X \in M^\theta}{\sup}\left\{E\left[-\frac{dQ}{dP} \cdot X\right]-\sum_{i=1}^d m_i - E[l(-X-m)]\right\}\\
&= \underset{m \in \R^d}{\sup}\left\{-\sum_{i=1}^d m_i + \underset{X \in M^\theta}{\sup} \left(E\left[-\frac{dQ}{dP} \cdot X\right]- E[l(-X-m)]\right)\right\}\\
&= \underset{m \in \R^d}{\sup}\left\{-\sum_{i=1}^d m_i+ \underset{W \in M^\theta}{\sup}\left( E\left[-\frac{dQ}{dP} \cdot (W-m)\right]- E[l(-W)]\right)\right\}\\
&= \underset{m \in \R^d}{\sup}\left\{-\sum_{i=1}^d m_i + E\left[\frac{dQ}{dP} \cdot m\right] + \underset{W \in M^\theta}{\sup}\left( E\left[-\frac{dQ}{dP} \cdot W\right]- E[l(-W)]\right)\right\}\\
&= \underset{m \in \R^d}{\sup}\left\{\sum_{i=1}^d -m_i + m_iE\left[\frac{dQ_i}{dP}\right] + \underset{W \in M^\theta}{\sup}\left( E\left[\frac{dQ}{dP} \cdot W\right]- E[l(W)]\right)\right\}\\
&= \underset{m \in \R^d}{\sup}\left\{ 0 + \underset{W \in M^\theta}{\sup} E\left[\frac{dQ}{dP} \cdot W - l(W)\right]\right\}\\
&= \underset{W \in M^\theta}{\sup} E\left[\frac{dQ}{dP} \cdot W - l(W)\right]
\end{align*}
Note that, for $W \in M^\theta$, we have for $Q \in \mathcal{D}^{\theta^*}$, $\frac{d Q}{d P} \cdot W \in L^1$, thanks to Fenchel inequality. Furthermore, since  $\sum W_i \le l(W) \le \theta(W)$ and both $\theta(W)$ and $\sum W_i$ are in $L^1$, we have $l(W) \in L^1$. This allows us to write in the lines above $E[\frac{dQ}{dP} \cdot W]- E[l(W)] = E\left[\frac{dQ}{dP} \cdot W - l(W)\right]$.\\
Now, we would like to interchange the expectation with the supremum. To this end, we use Corollary on page 534 of \citet{rockafellar1968integrals} with $L = M^\theta$, $L^* = L^{\theta^*}$ and $F(x) = l(x)$. Note that $l$ is a lower-semicontinuous proper convex function, and it is easy to verify that $M^\theta$ and $L^{\theta^*}$ are decomposable in their sens, so that all the conditions needed to apply this Corollary are satisfied. We get finally that,
$$R^*(-Q) = E\left[l^*\left(\frac{dQ}{dP}\right)\right] := \alpha(Q).$$
Finally, since $R(X)$ is finite, then the maximum can be taken over $\mathcal{D}^{l^*}$ instead of $\mathcal{D}^{\theta^*}$.
\end{itemize}
\end{proof}
\begin{Definition}
A risk allocation is any minimizer of \eqref{OCE}. When it is uniquely determined, we denote it $RA(X)$.
\end{Definition}
\begin{Theorem}\label{Characterization}
Let $l$ be a loss function. Then, for every $X \in M^\theta$, the set of risk allocations is non empty and bounded. Furthermore, risk allocations are characterized by the following first order condition:
\begin{equation}\label{charactEq}
1 \in E[\partial l(-X - m^*)].
\end{equation}
Moreover, the supremum in \eqref{Representation} is attained for $Z^*$ such that $Z^* \in \partial l(-X -m^*)$ a.s. and $E[Z^*]=1$.
\end{Theorem}
\begin{proof}
The arguments used in this proof are an extension of the univariate case. To prove that the set of risk allocations is non empty and bounded, it is sufficient to show that the objective function has no direction of recession thanks to Theorem 27.1(d) in \citet{rockafellarConvex}. Let $w \neq 0$ and let $f(w) \vcentcolon= \sum_{i=1}^d w_i + E[l(-X - w)]$. We have,
\begin{align*}
f0^+(w) &= \underset{r \to \infty}{\lim} \frac{f(m + r w) - f(m)}{r}\\
&= \underset{r \to \infty}{\lim} \frac{\sum m_i + r \sum w_i + E[l(-X - m - r w)]  - \sum m_i - E[l(-X - m)]}{r}\\
&= \sum w_i + \underset{r \to \infty}{\lim}\frac{E[l(-X - m - r w)] - E[l(-X - m)]}{r}\\
&= \sum w_i + \underset{r \to \infty}{\lim}\frac{E[l(-X - m - r w)]}{r}.\\
\end{align*}
Now, since $l$ is convex and $l(0) = 0$, for $\lambda > 1$ we have $\frac{1}{\lambda}l(\lambda x) \ge l(x)$. This implies, together with Lebesgue's dominated convergence theorem and lower-semicontinuity of $l$
\begin{align*}
f0^+(w) &\ge \sum w_i + \underset{r \to \infty}{\lim} E\left[ l \left(\frac{-X - m}{r}-w \right)\right]\\
&= \sum w_i + E\left[ \underset{r \to \infty}{\liminf}~ l \left(\frac{-X - m}{r}-w \right) \right]\\
&\ge \sum w_i + l(-w) > 0.
\end{align*}
The last strict inequality is a consequence of the third property of $l$. So we have shown that for every $w \neq 0$, $f0^+(w) > 0$, i.e., $f$ has no direction of recession. We conclude that the set of minimizers is non empty bounded set. Moreover, we have $m^* \in \mathrm{argmin} f$ if and only if $m^*$ satisfies $0 \in \partial f(m^*)$. Using Theorem 4.47 in \citet{shapiro}, we can interchange the partial operator and the expectation sign leading to the following characterization of minimizers:
$$ m^* ~\text{is a minimizer of}~f \Leftrightarrow 1 \in E[\partial l(-X - m^*)].$$
In the following, we prove that the maximum in \eqref{Representation} is attained for $Z^* \in \partial	l(-X-m^*)$ a.s. and $E[Z^*] = 1$. We start by proving the existence of such $Z^*$.
Let $m^*$ be such that $1 \in E[\partial	l(-X-m^*)]$. Note that, for each $z \in \R^d$, if $\nu \in \partial l(z)$, then $\nu$ is nonnegative. In fact, by definition, we have, $l(x) \ge l(z) + \sum_{i=1}^d \nu_i(x_i-z_i)$, $\forall x \in \R^d$. So, if for some $k \in {1,...,d}$, $\nu_k < 0$, then choosing $x = z - n e_k < z$ where $e_k$ is the $k$-th standard unit vector, we get that $-n \nu_k \le l(x) - l(z) \le 0$. By sending $n$ to $+\infty$, we get a contradiction. Therefore, since $1 \in E[l(-X-m^*)]$, there exists a random variable $Z^*$ such that $Z^*\ge 0$ and $Z^* \in \partial l(-X-m^*)$ a.s. and $E[Z^*]=1$.\\
Next, we will show that $Z^* \in \mathcal{D}^{\theta^*}$, that is $E[l^*(Z^*)] < \infty$. Note that since $Z^* \in \partial l(-X-m^*)$, we have that, 
\begin{equation}\label{subgradient}
l^*(Z^*)= Z^*\cdot (-X-m^*)-l(-X-m^*), ~a.s.
\end{equation}
First, we will start by proving that $Z^* \cdot X \in L^1$. Thanks to \eqref{subgradient}, we have $X \cdot Z^* + l^*(Z^*) = -m^* \cdot Z - l(-X-m^*)$. Because $X \in M^\theta$, the right term of the previous equality is in $L^1$. So, this shows that $X \cdot Z^* + l^*(Z^*) \in L^1$. Recall that $l^*(z) \ge 0$ for all $z\in \R^d$ so that we have $(X \cdot Z^*)^+ \in L^1$. It remains to show that $(X\cdot Z^*)^- \in L^1$. Using the convexity of $l$, we have the following inequality:
$$l(2(-X-m^*)) \ge l(-X -m^*) +Z^* \cdot (-X-m^*),~a.s.$$
This in turn implies that $X\cdot Z \ge l(-X-m^*)-l(2(-X-m^*))-Z^*\cdot m^*$. The RHS of this inequality is in $L^1$ as $X \in M^\theta$. Hence, we get that $(Z^* \cdot X)^- \in L^1$. We are now able to say that all the terms in the RHS of \eqref{subgradient} are in $L^1$. We conclude that $l^*(Z^*) \in L^1$. Moreover, we have,
\begin{align*}
E[-X \cdot Z^*] - E[l^*(Z^*)] &= E[-X\cdot Z^*]-E[Z^*\cdot(-X-m^*) - l(-X-m^*)]\\
&=E[-X\cdot Z^* -Z^*\cdot (-X-m^*) + l(-X-m^*)] = E[Z^*\cdot m^* + l(-X-m^*)]\\
&= R(X),
\end{align*}
where we used the optimality of $m^*$ in the last equality. This completes the proof.
\end{proof}
\begin{Example}
The following example with a bidimensional loss function of exponential type as in Example \ref{example}, that is:
$$l(x_1, x_2) = \frac{e^{\lambda_1 x_1}-1}{\lambda_1} + \frac{e^{\lambda_2 x_2}-1}{\lambda_2} +\alpha e^{\lambda_1 x_1 + \lambda_2 x_2},~\text{where}~\lambda_1 > 0,~\lambda_2 > 0,~\alpha \ge 0.$$
If $X \sim \mathcal{N}(0, \Sigma)$ with $\Sigma = \begin{pmatrix}
\sigma_1^2 & \rho \sigma_1 \sigma_2\\
\rho \sigma_1 \sigma_2 & \sigma_2^2
\end{pmatrix}$, then we can solve explicitly the optimal risk allocations in \eqref{charactEq} and to obtain
\begin{equation}\label{optimal}
m_i^* = \left\{
\begin{aligned}
 &\frac{\lambda_i \sigma_i^2}{2}, ~\text{if}~\alpha = 0,\\
 &\frac{\lambda_i \sigma_i^2}{2} - \frac{1}{\lambda_i}\ln(SC_{ij}),~ j \neq i,~\text{if}~\alpha > 0,
\end{aligned}\right.
\end{equation}
\end{Example}
where the term $SC_{ij}, i \neq j,$ is the positive solution to the following second order equation:
$$ \alpha \lambda_j \exp(\rho \sigma_i \sigma_j \lambda_i \lambda_j) X^2 + \left(1 + \alpha(\lambda_i - \lambda_j)\exp(\rho \sigma_i \sigma_j \lambda_i \lambda_j)\right)X - 1 = 0$$
The risk measure could also be derived in explicit form:
\begin{equation}\label{riskMeasure}
R(X) = m_*^1 + m_*^2 + \frac{2 - \alpha}{\lambda_1}(SC_{12} - 1)=m_*^1 + m_*^2 + \frac{2 - \alpha}{\lambda_2}(SC_{21} - 1) ,
\end{equation}
\begin{Remark}\label{remark}
The formula obtained in \eqref{optimal} is close to the one in Example 3.12 in \citet{armenti}. It shows that the optimal allocations are disentangled into two components: the first one is an individual contribution which takes the form of the entropic risk measure of $X_i$  and the second one is a systemic contribution which involves correlations between the two components of the system. This formula shows also an interesting feature: the partial differential of SRC with respect to $\rho$ is always positive. This can be interpreted in the following way: the more correlated the system is, the riskier is. Note that this is not true in general and depends on the loss function $l$ used.
\end{Remark}
\begin{Corollary}
Let $l$ a strictly convex loss function. Then,
$$RA(X+r) = RA(X)-\sum_{i=1}^d r_i,~ \text{for every}~ X \in M^\theta ~\text{and}~ r\in \R^d.$$
If $l$ is additionally positive homogeneous, then
$$RA(\lambda X) = \lambda RA(X), ~ \text{for every}~ X \in M^\theta ~\text{and}~ \lambda > 0.$$
\end{Corollary}
\begin{proof}
Let $X \in M^\theta$ and $r \in \R^d$. $m\vcentcolon RA(X+r)$ is the unique solution of $1 \in E[\partial l(-X - r - m)]$. Setting $w = r + m$, we obtain that $w$ satisfies $1 \in E[\partial l(-X - w)]$, which by uniqueness implies that $w = RA(X)$, that is $RA(X+r) = RA(X) - r$. Let $\lambda > 0$, we have,
\begin{align*}
RA(\lambda X) &= \underset{w}{\argmin}\left\{\sum w_i + E[l(-\lambda X - w)]\right\}\\
&=\underset{w}{\argmin}\left\{\sum_{i=1}^d w_i + \lambda E\left[l(-X - \frac{w}{\lambda})\right]\right\}\\
&= \underset{w}{\argmin}\left\{\sum_{i=1}^d \frac{w_i}{\lambda} + \lambda E\left[l(-X - \frac{w}{\lambda})\right]\right\}\\
&= \lambda~\underset{w}{\argmin}\left\{\sum_{i=1}^d w_i + E[l(-X - w)]\right\} = \lambda RA(X).
\end{align*}
\end{proof}
Now, we focus on the study of the sensitivity of our multivariate risk measure. We first give the definition of the marginal risk contribution of $Y \in M^\theta$ to $X\in M^\theta$.
\begin{Definition}
For $X, Y \in M^\theta$, we define the marginal risk contribution of $Y$ to $X$ as the sensitivity of the risk associated to $X$ when an impact $Y$ is applied as
\begin{equation}
R(X, Y) \vcentcolon = \underset{\epsilon \searrow 0}{\limsup} \frac{R(X +\epsilon Y) -R(X)}{\epsilon}.
\end{equation}
If $R(X + \epsilon Y)$ admits a unique risk allocation $RA(X +\epsilon Y)$ for small enough $\epsilon \ge 0$, then we define the risk allocation marginals of $X$ with respect to the impact of $Y$ as:
\begin{equation}
RA_i(X;Y) \vcentcolon \underset{\epsilon \searrow 0}{\limsup} \frac{RA_i(X + \epsilon Y) - RA_i(X)}{\epsilon},~i=1,...,d.
\end{equation}
\end{Definition}
\begin{Theorem}\label{sensitivity}
Let $X, Y \in M^\theta$ and assume that $l$ is differentiable. Then,
\begin{equation}\label{sensitivity1}
R(X, Y) = -E[Y\cdot \nabla l(-X -m^*)] = -\sum_{i=1}^d E_{Q^n_*}[Y^n],
\end{equation}
where $m^*$ is such that, $E[\nabla l(-X-m^*)]=1$, i.e. an infinimum for \eqref{OCE} and $\frac{d Q_*}{dP} \vcentcolon= \nabla l(-X-m^*)$.\\
If furthermore, $l$ is twice differentiable such that we can interchange the differentiation and expectation of $m \mapsto E[\nabla l(-X-m)]$ and $M \vcentcolon = E[\nabla^2 l(-X-m^*)]$ is invertible, then we have,
\begin{itemize}
\item There exists a unique $m_\epsilon$ optimum of $R(X+\epsilon Y)$ for small enough $\epsilon \ge 0$.
\item As a function of $\epsilon$, $m_\epsilon$ is differentiable and we have 
\begin{equation}\label{sensitivity2}
RA(X,Y) = M^{-1} V,~V \vcentcolon = -E[\nabla^2 l(-X-m^*)Y].
\end{equation}
\end{itemize}
\end{Theorem}
\begin{proof}
Take $X, Y \in M^\theta$ and let $m^*$ be an infinimum for $R(X)$. We have $R(X) = \sum_{i=1}^d m_i^* + E[l(-X-m^*)]$ and $E[\nabla l(-X-m^*)]=1$. By the definition of $R(X + \epsilon Y)$, we have
\begin{align*}
\frac{R(X + \epsilon Y) - R(X)}{\epsilon} &\le \frac{\sum m_i^* + E[l(-X - \epsilon Y - m^*)] - \sum m_i^* - E[l(-X-m^*)]}{\epsilon}\\
&=E\left[ \frac{l(-X - \epsilon Y - m^*) - l(-X-m^*)}{\epsilon} \right].
\end{align*}
Using the convexity, monotonocity of $l$ and the fact that $-l(x) \ge -\sum x_i, x \in\R^d$, for $0 < \epsilon < \frac{1}{2}$, we get that,
\begin{align*}
\frac{l(-X - \epsilon Y - m^*) - l(-X-m^*)}{\epsilon} &\le \frac{l(-X - m^* - (1-\epsilon)Y) - l(-X-m^*)}{1-\epsilon}\\
&\le \frac{l(|X| +|m^*| + (1-\epsilon)|Y|) - l(-X-m^*)}{1-\epsilon}\\
&\le 2 \left(l(|X| +|m^*| + |Y|) + \sum_{i=1}^d X_i + m_i\right).
\end{align*}
Since $X$ and $Y$ are in $M^\theta$, the last term is bounded from above by a random variable which is in $L^1$. Therefore, using Fatou's lemma, we obtain that,
$$\underset{\epsilon \searrow 0}{\limsup}~\frac{R(X + \epsilon Y) - R(X)}{\epsilon} \le E[-Y \cdot \nabla l(-X-m^*)].$$
Now, using the representation given in Theorem \ref{Characterization}  $R(X+\epsilon Y)=\underset{Q \in \mathcal{D}^{l^*}}{\max}~E_Q[-(X +\epsilon Y)]-E[l^*(\frac{dQ}{dP})]$,  and that $R(X) = E_{Q^*}[-X]-E[l^*(\frac{dQ^*}{dP})]$ with $\frac{d Q^*}{d P} = \nabla l(-X -m^*)$, we get, 
$$R(X+\epsilon Y) \ge E\left[-(X + \epsilon Y)\cdot\frac{dQ^*}{dP}-l^*\left(\frac{dQ^*}{dP}\right)\right]  =  R(X)- \epsilon E[Y\cdot \nabla l(-X -m^*)],$$
Consequently, the other inequality follows:
$$ \underset{\epsilon \searrow 0}{\limsup} \frac{R(X + \epsilon Y) - R(X)}{\epsilon} \ge -E[Y \cdot \nabla l(-X-m^*)].$$
Second assertion is a direct application of Theorem 6 pp 34 in \citet{fiacco1990nonlinear}.
\end{proof}
In the following Corollary, we explicit the impact of an independent exogenous shock in the case $X$ and $Y$ are independent.
\begin{Corollary}
If $X$ and $Y$ are independent, then under assumptions of Theorem \ref{sensitivity}, we have,
\begin{equation}\label{sensitivity3}
R(X, Y) = -\sum_{i=1}^d E[Y_i],~~~RA(X, Y) = -E[Y].
\end{equation}
\end{Corollary}
\begin{Remark}\
\begin{enumerate}
\item The equations in \eqref{sensitivity1} and  \eqref{sensitivity2} are very interesting and show the relevance of the dual optimizer $Q_*$. More precisely, \eqref{sensitivity1} shows that the marginal risk contribution can be quantified thanks to the optimal probability $Q_*$.
\item If only the value of portfolio $i$ changes by a cash amount, that is $Y^i = c^i$ and $Y^j = 0$ for $j \neq i$, then the marginal risk contribution $R(X, Y) = -c^i$ is exactly covered by the marginal risk allocation $RA_i(X, Y) = -c^i$ of portfolio $i$, whereas marginal risk allocations of other portfolios remain unchanged, i.e. $RA_j(X, Y) = 0$ for $j \neq i$. This property of full responsibility for one's own changes in financial position is known as causal responsibility (see \citet{brunnermeier2019measuring}). In general, this is no longer true if $Y^i$ is a random variable, but in the particular case when $Y^i$ is independent of $X$, this property remains true as suggested by equation \eqref{sensitivity3}.
\item Equation \eqref{sensitivity3} shows an interesting feature: Assume that two institutions $i \neq j$ change their positions in opposite direction, that is $Y^i = -Y^j$, then the marginal risk contribution is zero, as if the portfolios compensate each other and a risk sharing mechanism take place.
\end{enumerate}
\end{Remark}
\begin{Example}
In this example, we illustrate the impact of an exogenous shock that may depend on $X$. More specifically, we consider a system with two portfolios $X=(X_1, X_2)$, an exogenous shock $Y=(Y_1, 0)$ impacting the first component only and a loss function of exponential type as in \eqref{exponential}: 
$$l(x_1, x_2)=\frac{e^{\lambda_1 x_1}-1}{\lambda_1} + \frac{e^{\lambda_2 x_2}-1}{\lambda_2} + \alpha e^{\lambda_1 x_1 + \lambda_2 x_2}.$$
As per Theorem \ref{Characterization}, there exists a unique risk allocations $m^*$. To alleviate the expressions, we denote the following:
\begin{equation*}
\left\{
\begin{aligned}
&C_{X_1} \vcentcolon = E[e^{\lambda_1(-X_1-m_1^*)}],~C_{X_2} \vcentcolon = E[e^{\lambda_2(-X_2-m_2^*)}],\\
&C_X \vcentcolon= E[e^{\lambda_1(-X_1-m_1^*) +\lambda_2(-X_2-m_2^*)}],\\
&C_{X_1Y} \vcentcolon = E[Y_1e^{\lambda_1(-X_1-m_1^*)}],~C_{X_2Y} \vcentcolon = E[Y_1e^{\lambda_2(-X_2-m_2^*)}],\\
&C_{XY} \vcentcolon = E[Y_1e^{\lambda_1(-X_1-m_1^*) +\lambda_2(-X_2-m_2^*)}].
\end{aligned}\right.
\end{equation*}
The matrix $M$ and vector $V$ in Theorem \ref{sensitivity} can be expressed thanks to the quantities above after some simple but lengthy computations (omitted here):
$$ M = \begin{pmatrix}
\lambda_1C_{X_1} + \alpha \lambda_1^2 C_X & \alpha \lambda_1\lambda_2 C_X\\
\alpha \lambda_1\lambda_2 C_X & \lambda_2C_{X_2} + \alpha \lambda_2^2 C_X
\end{pmatrix}, ~V = \begin{pmatrix}
-\lambda_1 C_{X_1Y} - \alpha \lambda_1^2 C_{XY} \\
-\alpha \lambda_1 \lambda_2 C_{XY}
\end{pmatrix}.$$
The risk contribution marginal and risk allocations marginals follows:
\begin{align}
R(X, Y) &= -C_{X_1Y} - \alpha \lambda_1 C_{XY},\\
RA(X, Y) &= \text{Common Factor}\times \begin{pmatrix}
-C_{X_2}C_{X_1Y}-\alpha (\lambda_1 C_{X_2}C_{XY}+\lambda_2C_XC_{X_1Y}))\\
 -\alpha(\lambda_1 C_{X_1}C_{XY} - \lambda_1C_XC_{X_1Y}).
\end{pmatrix}
\end{align}
We notice the following:
\begin{itemize}
\item $R(X, Y)$ is disentangled into two components. The first one is the contribution of the first component in the risk contribution marginal and the second is a systemic contribution that is proportional to $\alpha$. This same remark holds for $RA_1(X, Y)$.
\item The asymmetry of the shock on $X_1$ can be seen in the systemic contribution in $RA_1(X, Y)$ and $RA_2(X, Y)$. Indeed, we notice that both components are impacted by the shock and this is reflected by the term $-\alpha \lambda_1 C_{X_2}C_XY$ for the first component and $-\alpha\lambda_1 C_{X_1} C_XY$ for the second. However, there is a correction term proportional to $\lambda_2$ that is subtracted from the first component whereas another correction term proportional to $\lambda_1$ is added to the second component.
\item In the case $\alpha =0$, i.e. without a systemic component, the risk marginal of the second portfolio is zero. This something we would expect as we applied a shock only on the first component. In other words, the first component takes full responsibility in this case.
\end{itemize}
\end{Example}
In the rest of the paper, for every $X \in M^\theta$, we will assume the following:
\begin{enumerate}[label=($\mathcal{A}$\textsubscript{l})]
\item
\begin{enumerate}[label = \roman*.]
\item \label{AlOCE1} For every $m_0 \in \R^d$, $m \mapsto l(-X-m)$ is differentiable at $m_0$ a.s.;
\item \label{AlOCE2} $m \mapsto E[l(-X-m)]$ is strictly convex.
\end{enumerate} 
\label{AlOCE}
\end{enumerate}
Under assumption \ref{AlOCE}, there exists a unique risk allocation $m^*$ that is characterized through the following equation:
$$1 = E[\nabla l(-X-m^*)].$$
\section{Computational aspects}\label{ComputationalAspects}
In this section, we develop numerical schemes to compute the optimal risk allocations $m^*$ and $R(X)$ using stochastic algorithms (SA).
This is because the optimal allocations are solutions of a convex optimization problem whose objective function can be expressed as an expectation. Stochastic algorithms are generally used to find zeros of a certain function $h$ that is unknown but could be approximated using some estimate. More specifically, SA algorithms take the following form: $Z_{n+1} = Z_n \pm \gamma_n Y_n$, where $Y_n$ is a noisy estimate of $h(Z_n)$ and $(\gamma_n)$ is a step sequence decreasing toward $0$. This algorithm is known as Robbins-Monro algorithm (RM). For an overview of SA algorithms, we refer to \citet{duflo}. However, in order to be able to use classical convergence results of SA, we need a sub-linear growth over the function $h$ (see for example condition (8) of Theorem 2.2 in \citet{pages}), which in our case, considerably limits the choice of loss functions. To circumvent this condition, we will use a \q{constrained} variant where we force the iterations of the (RM) algorithm to remain in a certain compact $K$ set that contains the optimal allocations. One could also use the well-known  projection \q{\`a la Chen} algorithm based on reinitializations of the algorithm and taking larger compact sets each time the iteration goes out of the compact set (cf. \citet{chen1986stochastic}). For the sake of simplicity, we will use the classical \q{constrained} version with a fixed compact set $K$ as it has the same asymptotic behaviour as the one with projection \q{\`a la Chen}. In \citet{armenti}, numerical schemes were developed to find optimal allocations for multivariate shortfall risk measures. They first estimated the different expectations using Monte Carlo/Fourier methods and then  a root finding algorithm was used to find the optimum. Although this method shows good results of convergence and is quite fast, it has several drawbacks: It is sensitive to the starting point of the root finding algorithm and one has no control over the error of estimation. With SA, there is one major advantage over the former method: One could derive Central Limit Theorems (CLT) for the estimation and therefore obtain confidence intervals could be obtained for the estimators.\\
We will study the behaviour of SA algorithms for the different loss functions in example \ref{example}. Recall that, for $X \in M^\theta$ and under the assumption \ref{AlOCE}, there exists a unique risk allocation $m^*$ solution of $1 = E[\nabla l(-X-m^*)]$. We fix $K$ a hyperrectangle such that $m^* \in \mathrm{int}(K)$ and we define for $X\in M^\theta$ and $m \in \R^d$:
\begin{equation}
\left\{
\begin{aligned}
H_1(X, m) &\vcentcolon=  \nabla l(-X-m) - 1,\\
~h_1(m) &\vcentcolon= E[H_1(X,m)],\\
\sigma^2(m) &\vcentcolon= E[||H_1(X, m)-h_1(m)||^2],\\
m^{2 + p}(m) &\vcentcolon= E[||H_1(X, m)-h_1(m)||^{2 +p}||,~p > 0,\\
\Sigma(m) &\vcentcolon= E[(H_1(X, m)-h_1(m))(H_1(X,m)-h_1(m))^\intercal].
\end{aligned}\right.
\end{equation}
We introduce the following set of assumptions:
\begin{center}
\begin{minipage}{11cm}
\begin{enumerate}[label=($\mathcal{A}$\textsubscript{a.s.})]
\item 
\begin{enumerate}[label=\roman*.]
\label{Aas}
\item\label{Aas1} $\sum_{n \ge 0} \gamma_n = + \infty~\text{and}~\sum_{n \ge 0} \gamma_n^2 < \infty$;
\item\label{Aas2} $h_1$ is continuous on $K$;
\item\label{Aas3} $\underset{m \in K}{\sup}~\sigma^2(m) < \infty$.
\end{enumerate}
\end{enumerate}
\end{minipage}
\end{center}
\begin{Theorem}\label{as}
Let $(X_n)$ a sequence of random variables having the same law as $X \in M^\theta$ and define the sequence $(m_n)$ as follows:
\begin{equation}\label{SA}
m_{n+1} = \Pi_K\left[m_n + \gamma_n H_1(X_{n+1},m_n)\right],~ m_0 \in L^0,
\end{equation}
where $\Pi_K$ is the projection into $K$. Under \ref{AlOCE} and \ref{Aas}, we have, $m_n \to m^*$ a.s. as $n \to \infty$.
\end{Theorem}
\begin{proof}
Following the same arguments of Theorem 2.4 in \citet{hal2}, the only limit point of the projected ODE associated to the algorithm in \eqref{SA} is $m^*$. Thus, we can use Theorem 2.1 in \citet{kushner} that argues that $m_n$ will converge to the limit point $m^*$ if we can verify their assumptions (A2.1)-(A2.5). Indeed, (A2.1) is guaranteed thanks to the assumption \ref{Aas}-\ref{Aas3}. The other assumptions are verified thanks to \ref{Aas}-\ref{Aas2}.
\end{proof}
Once we have an estimator of $m^*$, it comes the question of estimating the multivariate OCE $R(X) = \sum_{i=1}^d m_i^* +E[l(-X - m^*)]$. A naive way consists in estimating $R(X)$ in a two steps procedure:
\begin{itemize}
\item Step 1: Use the estimate $m_n$ from \eqref{SA} to have a good approximation of $m^*$.
\item Step 2: Use another sample of $X$ to approximate $R(X)$ using Monte Carlo:
\begin{equation}\label{MC}
R(X) \approx \sum_{i=1}^d m^* + \frac{1}{n} \sum_{k=1}^n l(-X_k -m^*).
\end{equation}
\end{itemize}
A natural way to avoid this two steps procedure is to use a companion procedure (CP) of the algorithm \eqref{SA} and to replace the quantity $m^*$ in \eqref{MC} by its estimate at step $k-1$, that is,
$$R_n = \frac{1}{n}\sum_{k=1}^n \left((\sum_{i=1}^d m_{k-1}^i) + l(-X_k - m_{k-1})\right).$$
Note that $R_n$ is a sequence of empirical means of non i.i.d. random variables that can be written also as:
\begin{equation}
R_{n+1} = R_{n} - \frac{1}{n + 1}H_2(X_{n+1}, R_{n}, m_n),~n \ge 0,~R_0 \in L^0,
\end{equation}
where $$H_2(x, R, m) \vcentcolon= R - F(x, m) \vcentcolon= R - \left(\sum_{i=1}^d m^i + l(-x -m)\right).$$
We are now facing two procedures with different time steps: one for the estimation of $m^*$ and the other one for the estimation of $R(X)$. In the following theorem, we prove the consistency of the second procedure using the same time step as the first one $(\gamma_n)$, namely,
\begin{equation}\label{rho_n}
R_{n+1} = R_n -\gamma_n H_2(X_{n+1}, R_n, m_n),~n \ge 0,~R_0 \in L^0.
\end{equation}
To this purpose we need the following assumption:
\begin{center}
\begin{minipage}{17cm}
\begin{enumerate}[label=($\mathcal{A}$\textsubscript{CP})]
\item 
\label{Aas'}
$\forall m \in K, ~l(-X-m) \in L^2$ and $m \to E[|l(-X-m)|^2]$ is bounded around $m^*$.
\end{enumerate}
\end{minipage}
\end{center}
\begin{Theorem}
\label{Rn}
Assume that assumptions \ref{AlOCE}, \ref{Aas} and \ref{Aas'} hold and let $(m_n)$ be given by \eqref{SA} and $(R_n)$ by \eqref{rho_n}. Then $R_n \to R(X)$ a.s.
\end{Theorem}
\begin{proof} 
For $n \in \N$, define the sequence $(S_n)$ as:
$$S_n = \frac{1}{\prod_{k=0}^{n-1} (1 - \gamma_k)},~ S_0 = 0.$$
We have, 
\begin{equation}\label{S_n}
S_{n+1} = \frac{S_{n}}{1 - \gamma_{n}}=S_{n}\left(1+\frac{\gamma_{n}}{1-\gamma_{n}}\right)=S_{n}+\gamma_{n} S_{n+1}.
\end{equation}
Therefore using \eqref{rho_n}, we have, 
\begin{align*}
S_{n+1}R_{n+1} &= S_{n+1} (R_n -\gamma_n H_2(X_{n+1}, R_n, m_n))\\
&= S_n R_n + \gamma_n S_{n+1} R_n -\gamma_n S_{n+1} H_2(X_{n+1}, R_n, m_n)\\
&= S_n R_n + \gamma_n S_{n+1} R_n - \gamma_n S_{n+1} R_n +\gamma_n S_{n+1} F(X_{n+1},m_n)\\
&=S_n R_n + \gamma_n S_{n+1} F(X_{n+1},m_n).
\end{align*}
This implies for $n \in \N^*$,
\begin{equation}\label{inter1}
R_n = \frac{1}{S_n}R_0 + \frac{1}{S_n}\sum_{k=0}^{n-1}\gamma_k S_{k+1}F(X_{k+1},m_k).
\end{equation}
First, we have $$\log(S_n) = - \sum_{k=0}^{n-1} \log(1 -\gamma_k) \ge \sum_{k=0}^{n-1}\gamma_k,$$
and since by assumption, the RHS of the last inequality goes to $\infty$ as $n \to \infty$, we deduce that $S_n \to  \infty$ as $n \to \infty$ and we get immediately that the first term of the RHS of \eqref{inter1} goes to $0$ as $n$ goes to $\infty$.
Rewriting \eqref{inter1} by introducing $f(m) \vcentcolon= E[F(X, m)]$ and the martingale difference sequence $\delta M_{k+1} = F(X_{k+1}, m_k) - f(m_k)$ with respect to the natural filtration $\mathcal{F}_k \vcentcolon=\sigma(m_0, X_1,...,X_k)$, we obtain,
\begin{equation}\label{inter2}
R_n =  \frac{1}{S_n}R_0 + \frac{1}{S_n}\sum_{k=0}^{n-1}\gamma_k S_{k+1} \delta M_{k+1} + \frac{1}{S_n}\sum_{k=0}^{n-1}\gamma_k S_{k+1} f(m_k).
\end{equation}
Thanks to \eqref{S_n}, we have $\sum_{k=0}^{n-1} \gamma_k S_{k+1} = S_n$. Because $f$ is convex (assumption \ref{AlOCE}-\ref{AlOCE2}) and therefore continuous at $m^*$, Cesaro's Lemma implies that 
the third term in the previous equality converges to $f(m^*) = R(X)$. The a.s. convergence of $R_n$ will follow from the a.s. convergence of the second term toward $0$. Indeed, let us denote,
$$M_n^\gamma = \sum_{k=1}^{n}\gamma_{k-1} \delta M_{k}$$
Note that $(M_n^\gamma)$ is a $\mathcal{F}$-martingale such that 
$$\langle M^\gamma \rangle_\infty = \sum_{n=0}^\infty \gamma_n^2 E[|\delta M_n|^2 | \mathcal{F}_{n-1}].$$
But we also have,
$$E[|\delta M_n|^2|\mathcal{F}_{n-1}] \le E[|l(-X-m)|^2]_{|m = m_{n-1}},$$
and assumption \ref{Aas'} implies that 
$$\sup_{n\ge 1}E[|\delta M_n|^2|\mathcal{F}_{n-1}] < \infty,~a.s.$$
Using the martingale convergence theorem, we get that $(M_n^\gamma)$ converges to some random variable. Finally, by Kronecker's Lemma we deduce that the second term of \eqref{inter2} converges to $0$. This completes the proof.
\end{proof}
The step sequence in \ref{Aas}-\ref{Aas1} is typically of the following form $\gamma_n = \frac{c}{n^\gamma}$, where $\gamma \in (\frac{1}{2}, 1]$ and $c$ is a positive constant. The choice of $c$ plays a key role in the rate of convergence of SA algorithms. In order to circumvent problems related to the specification of the constant $c$, which are classical, we will use \q{averaging} techniques introduced by \citet{ruppert} and \citet{polyak}. We introduce the following assumptions:
\begin{enumerate}[label=($\mathcal{A}$\textsubscript{a.n.})]
\item
\begin{enumerate}[label = \roman*.]
\label{Aan}
\item\label{Aan1} $h_1$ is continuously differentiable and let $A \vcentcolon= Dh_1(m^*)$;
\item\label{Aan2} $(H_1(X_{n+1},m_n) \mathbf{1}_{|m_n - m^*| \le \rho})$ is uniformly integrable for small $\rho > 0$;
\item\label{Aan3} For some $p > 0$ and $\rho > 0, \underset{|m-m^*| \le \rho}{\sup} m^{2+p}(m) < \infty$;
\item\label{Aan4} $\Sigma(\cdot)$ is continuous at $m^*$ and  $\Sigma^*\vcentcolon= \Sigma(m^*)$ is positive definite.
\end{enumerate}
\end{enumerate}
The next theorem states the rate convergence of the average of the iterates of (RM) algorithm. 
\begin{Theorem}
Assume $\gamma \in (\frac{1}{2}, 1)$ and that assumptions \ref{AlOCE}, \ref{Aas} and \ref{Aan} hold. For any arbitrary $t > 0$, we define the (PR) sequence $(\overline{m}_n)$ as:
\begin{equation}\label{PR}
\overline{m}_n \vcentcolon = \frac{\gamma_n}{t}\sum_{i=n}^{n + t/\gamma_n - 1} m_i,
\end{equation}
where any upper summation index is interpreted as its integer part. Then, we have
\begin{equation}
\sqrt{\frac{t}{\gamma_n}}(\overline{m}_n - m^*) \to \mathcal{N}\left(0, A^{-1}\Sigma^*(A^{-1})^\intercal + O\left(\frac{1}{t}\right)\right).
\end{equation}
\end{Theorem}
\begin{proof}
This is a consequence of Theorem 1.1 chapter 11 page 377 in \citet{kushner} if we could verify their assumption (A1.1). Thanks to Theorem 2.1 of chapter 10 in \citet{kushner}, the condition (A1.1) is satisfied as soon as their conditions (A2.0)-(A2.7) hold. Assumption (A2.0) is automatically verified. (A2.1) is satisfied by \ref{Aan}-\ref{Aan2}. (A2.2) is a consequence of Theorem \ref{as}. (A2.4) follows from Taylor's expansions and \ref{Aan}-\ref{Aan1}. (A2.5) follows from the fact that $h_1(m^*) = 0$.  (A2.6) is satisfied since  $m^*$ is the optimum of a convex optimization problem. The first part and second parts of (A2.7) are guaranteed thanks to \ref{Aan}-\ref{Aan3} and \ref{Aan}-\ref{Aan4}. Finally, (A2.3) follows easily from Theorem 4.1 chapter 10 page 341 in \citet{kushner} since all their assumptions (A4.1)-(A4.5) are satisfied.
\end{proof}
\begin{Remark}
The previous CLT theorem states that under suitable conditions our average sequence is asymptotically normal with a corresponding covariance matrix that depends on $\Sigma^*$ and $A$. These quantities are unknown to us because,  first, in general,  they cannot be expressed in a closed form and second they depend on the optimum $m^*$. So, in practice, these two quantities need to be approximated in order to derive confidence intervals. In the following proposition, we provide consistent estimators of these two quantities.
\end{Remark}
\begin{Proposition}
Suppose \ref{AlOCE}, \ref{Aas} and \ref{Aan} hold. If $m \to E[||H_1(X, m)||^4]$ is bounded around $m^*$, then,
\begin{equation}\label{sigmaEst}
\Sigma_n \vcentcolon= \frac{1}{n}\sum_{k=0}^{n-1} H_1(X_{k+1}, m_k)^\intercal H_1(X_{k+1}, m_k) \to \Sigma^*~\text{a.s. as}~n \to \infty.
\end{equation}
Let $A_n^\epsilon$ be the matrix whose elements $A_n^\epsilon(i,j)$ for $i,j \in \{1,...,d\}$ are defined as follows:
$$ A_n^\epsilon(i,j) \vcentcolon= \frac{1}{\epsilon n}\sum_{k=0}^{n-1} H_1^i(X_{k+1}, m_k + \epsilon e_j) - H_1^i(X_{k+1}, m_k),$$
then, 
\begin{equation}\label{jacEst}
\underset{\epsilon \to 0}{\lim}~\underset{n \to \infty}{\lim} A_n^\epsilon = A~a.s.
\end{equation}
\end{Proposition}
\begin{proof}
The proof of this proposition relies mainly on the martingale convergence theorem. Let $(\delta M_k)_{k \in \N}$ be the sequence defined as:
$$\delta M_k \vcentcolon=H_1(X_{k+1}, m_k)^\intercal H_1(X_{k+1}, m_k) - \Sigma(m_{k}) - h_1(m_k)^\intercal h_1(m_k), ~k \ge 0.$$
$(\delta M_k)_{k \ge 0}$ is a martingale difference sequence adapted to $\mathcal{F}$ and therefore the following sequence $(M_k)_{k \in \N^*}$ defined as:
$$ M_k \vcentcolon= \sum_{i=1}^{k} \frac{\delta M_{i}}{i},~ k \ge 1,$$
is a $\mathcal{F}$-martingale. Furthermore, the boundedness of $m \to E[||H_1(X,m)||^4]$ around $m^*$, assumptions \ref{Aas}-\ref{Aas2} and \ref{Aan}-\ref{Aan4} imply that:
$$\underset{k\ge 1}{\sup}E[||\delta M_n||^2|\mathcal{F}_n] < \infty~a.s.$$
Consequently, the martingale convergence theorem implies the existence of a finite random variable $M_\infty$ such that $M_n \to M_\infty$. We then apply Kronecker's Lemma to get that $\frac{1}{n}\sum_{k=0}^{n-1} \delta M_{k+1} \to 0$. Since,
$$\Sigma_n = \frac{1}{n}\sum_{k=0}^{n-1} \delta M_{k} + \frac{1}{n} \sum_{k=0}^{n-1}\Sigma(m_k) + \frac{1}{n}\sum_{k=0}^{n-1} h_1(m_k)^\intercal h_1(m_k),$$
we deduce that $\Sigma_n \to \Sigma^*$. The proof of \eqref{jacEst} follows using the same arguments above.
\end{proof}
\begin{Remark}
\begin{enumerate}\
\item Instead of averaging on all observations for the estimators above, we could average using only recent ones. This might improve the behaviour of these estimators.
\item If we denote $V_n^\epsilon = (A_n^\epsilon)^{-1}\Sigma_n ((A_n^\epsilon)^{-1})^\intercal$, then we can obtain an approximate confidence interval for $m^*$ with a confidence level of $1-\alpha$ in the following form:
\begin{equation}
\left[\overline{m}^j_n -\sqrt{\frac{V_n^{\epsilon,jj}}{t c n^\gamma}}q_\alpha, \overline{m}^j_n -\sqrt{\frac{V_n^{\epsilon,jj}}{t c n^\gamma}}q_\alpha\right],~j\in \{1,...,d\}, \gamma \in (\frac{1}{2}, 1),
\end{equation}
where $q_\alpha$ is the $1-\frac{\alpha}{2}$ quantile of a standard random variable.
\end{enumerate}
\end{Remark}
\section{Numerical Analysis and Examples}
\label{numericalAnalysis}
In this section, we analyze and test the numerical methods developed in the previous section for the estimation of optimal allocations given by \eqref{PR} and risk measures given by \eqref{rho_n}. The implementation was done on a standard computer using Python and we write CT for computational time. All the computations were run on a standard laptop with a processor Intel(R) Core(TM) i7-9850H CPU @ 2.60GHz. The common parameters used in the computations are summarized in the following table:
\begingroup
\setlength{\tabcolsep}{10pt} % Default value: 6pt
\renewcommand{\arraystretch}{1.2}
\begin{table}[H]
\centering
\begin{tabular}{|c|c|}
\hline
Parameters & Values\\
\hline
$n$ & $500000$\\
\hline
$\gamma$ & $0.8$\\
\hline
$t$ & $10$\\
\hline
$c$ & $1$\\
\hline
$\epsilon$ & $10^{-6}$\\
\hline
\end{tabular}
\caption{Set of common parameters.}
\end{table}
\endgroup
\subsection{A first example}
We start here by estimating optimal allocations and multivariate OCE for the first loss function in \eqref{exponential}, that is:
$$l(x) = \sum_{i=1}^d \frac{e^{\lambda_i x_i}-1}{\lambda_i} + \alpha e^{\sum_{i=1}^d \lambda_i x_i}, ~\lambda_i > 0,~\alpha \ge 0.$$
We denote by $\lambda$ the vector of $\lambda_i$, $i \in \{1,...,d\}$. 
First, we test our algorithms in the case $d=2$ and the vector $X$ having a Gaussian distribution, as optimal allocations are expressed in a closed form in this case (see \eqref{optimal}). This will allow us to test the efficiency of our algorithms. Three cases are considered: In the first case, we take $\alpha = 0$ and  $\lambda = (1, 2)$, which as previously mentioned, corresponds to the entropic risk measure, a second one with $\alpha = 1$ and $\lambda = (1, 1)$, and finally a third one with $\alpha = 1$, $\lambda = (1, 2)$. As for the parameters for the normal distribution of $X$, we fix $\sigma_1 = \sigma_2 = 1$ and we take $\rho \in \{-0.9, -0.5, 0, 0.5, 0.9\}$ for each of the three cases. The compact set $K$ was set to $[0, 3]^2$ and the initial term $m_0 = (0, 0)$. 
\begingroup
\setlength{\tabcolsep}{3pt} % Default value: 6pt
\renewcommand{\arraystretch}{1.2}
\begin{table}[H]
\centering
\begin{tabular}{|c|c|c|c|c|c|c|c|c|c|}
\hline
$\rho$ & $R_n$ &$\overline{m}^1_n$ & $\overline{m}^2_n$ & CI1 & CI2 & $R(X)$& $m_*^1$ & $m_*^2$ & CT(s)\\
\hline
$-0.9$ & $1.5133$ & $0.4987$ &$0.9983$  & $[0.4945, 0.5030]$ & $[0.9874, 1.0093]$ & $1.5$ & $0.5$& $1$ & $68.8087$\\
$-0.5$ & $1.5220$ & $0.4964$ & $1.0010$ &$[0.4922, 0.5007]$ &$[0.9908, 1.0112]$ & $1.5$ & $0.5$ & $1$& $68.5388$ \\
$0$ &$1.5054$ & $0.4999$ & $1.0022$ & $[0.4956, 0.5042]$ & $[0.9888, 1.0156]$& $1.5$ & $0.5$ & $1$ & $69.1064$\\
$0.5$ &$1.5147$ & $0.5049$ & $0.9906$ & $[0.5006, 0.5092]$ & $[0.9803, 1.0009]$ &$1.5$  & $0.5$ & $1$ & $70.0251$\\
$0.9$ & $1.5264$ & $0.5031$ &$0.9970$ & $[0.4988, 0.5074]$& $[0.9867, 1.0073]$ & $1.5$ & $0.5$ & $1$ & $70.8549$\\
\hline
\end{tabular}
\caption{Numerical results: $\alpha = 0$ and $\lambda = (1, 2)$.}
\label{1stCase}
\end{table}
\endgroup
The table above summarizes the numerical results for the first case. The two columns CI1 and CI2 represent the confidence intervals of the (PR) estimators with a confidence level of $95\%$. Since $\alpha = 0$, the exact optimal allocations do not depend on the correlation coefficient $\rho$. This explains why we obtain the same values for $m^*$ for different values of $\rho$. The same remark goes for the estimators $\overline{m}_n$. Since $\lambda_2 > \lambda_1$, we expect as per formula \eqref{optimal} that $m_*^2 > m_*^1$. These numerical results suggest that the (PR) estimators $\overline{m}_n$ as well as the (RM) estimator $R_n$ approximate very well the exact optimal allocations $m^*$ and the risk measure $R(X)$. The width of the first confidence intervals (resp. second confidence intervals) is approximately $0.008$ (resp. $0.02$) which gives an accuracy of $1.5\%$ (resp. $2\%$) for the first estimator $\overline{m}_n^1$ (resp. $\overline{m}_n^2$).
\begingroup
\setlength{\tabcolsep}{3pt} % Default value: 6pt
\renewcommand{\arraystretch}{1.2}
\begin{table}[H]
\centering
\begin{tabular}{|c|c|c|c|c|c|c|c|c|c|}
\hline
$\rho$ & $R_n$ &$\overline{m}^1_n$ & $\overline{m}^2_n$ & CI1 & CI2 & $R(X)$ &$m_*^1$ & $m_*^2$ & CT(s)\\
\hline
$-0.9$ & $1.3139$ & $0.7689$ & $0.7704$  & $[0.7651, 0.7728]$ & $[0.7665, 0.7742]$ & $1.3036$ & $0.7702$& $0.7702$ & $69.1816$\\
$-0.5$ & $1.4198$ & $0.8511$ & $0.8552$ & $[0.8471, 0.8550]$ & $[0.8512, 0.8592]$ & $1.4105$ & $0.8545$ & $0.8545$& $73.6128$ \\
$0$ &$1.5897$ & $0.9827$ & $0.9801$ & $[0.9782,  0.9873]$ & $[0.9755, 0.9846]$& $1.5804$ &$0.9812$ & $0.9812$ & $69.8510$\\
$0.5$ & $1.8171$ & $1.1368$ & $1.1280$ & $[1.1307, 1.1430]$ & $[1.1220, 1.1339]$ & $1.7928$ & $1.1301$ & $1.1301$ & $71.4592$\\
$0.9$ & $2.0305$ & $1.2697$ &$1.2651$ & $[1.2612, 1.2782]$& $[1.2568, 1.2734]$& $1.9932$ & $1.2636$ & $1.2636$ & $73.0060$\\
\hline
\end{tabular}
\caption{Numerical results: $\alpha = 1$ and $\lambda = (1, 1)$.}
\label{2ndCase}
\end{table}
\endgroup
When taking the same values for $\lambda_1$ and $\lambda_2$, the system becomes symmetric and we obtain the same optimal allocations for the first and second component. We also notice that optimal allocations and their estimators increase with the correlation coefficient $\rho$ as it was expected from remark \ref{remark}. Again, the estimators approximate well the optimal allocations and the risk measure. The accuracy of all confidence intervals is around $\approx 1 \%$.
\begingroup
\setlength{\tabcolsep}{3pt} % Default value: 6pt
\renewcommand{\arraystretch}{1.2}
\begin{table}[H]
\centering
\begin{tabular}{|c|c|c|c|c|c|c|c|c|c|}
\hline
$\rho$ & $R_n$ &$\overline{m}^1_n$ & $\overline{m}^2_n$ & CI1 & CI2 & $R(X)$ &$m_*^1$ & $m_*^2$ & CT(s)\\
\hline
$-0.9$ & $1.6477$ & $0.6194$ & $1.1275$  & $[0.6152, 0.6237]$ & $[1.1184, 1.1366]$ & $1.6354$ & $0.6202$& $1.1285$ & $79.6036$\\
$-0.5$ & $1.7734$ & $0.7045$ & $1.2366$ & $[0.7001, 0.7089]$ & $[1.2280, 1.2452]$ & $1.7544$ & $0.7071$ & $1.2344$& $81.1432$ \\
$0$ &$2.0148$ & $0.8479$ & $1.4449$ & $[0.8421,  0.8538]$ & $[1.4309,  1.4588]$& $1.9943$ &$0.8465$ & $1.4406$ & $76.7199$\\
$0.5$ & $2.3749$ & $0.9922$ & $1.7260$ & $[0.9844, 1.0001]$ & $[1.7044, 1.7476]$ & $2.3354$ & $0.9859$ & $1.7344$ & $80.4979$\\
$0.9$ & $2.7790$ & $1.0812$ &$2.0416$ & $[1.0710,  1.0914]$& $[1.9981, 2.0850]$& $2.6652$ & $1.0728$ & $2.0285$ & $81.1827$\\
\hline
\end{tabular}
\caption{Numerical results: $\alpha = 1$ and $\lambda = (1, 2)$.}
\label{3rdCase}
\end{table}
\endgroup
In this final case, we take different values for $\lambda_1$ and $\lambda_2$. Table \ref{3rdCase} shows that the optimal allocations can be well approximated by the estimator in \eqref{PR}. This is also the case for the estimator $R_n$. Again, the optimal allocations as well as the risk measure increase with the correlation coefficient (see Figure \ref{loss1}).  All confidence intervals have an accuracy between $1\%$ and $2\%$ except the second confidence interval in the case $\rho = 0.9$ where the accuracy is a bit higher and is approximately around  $4\%$.
\begin{figure}[H]
    \centering
    \includegraphics[scale=0.9]{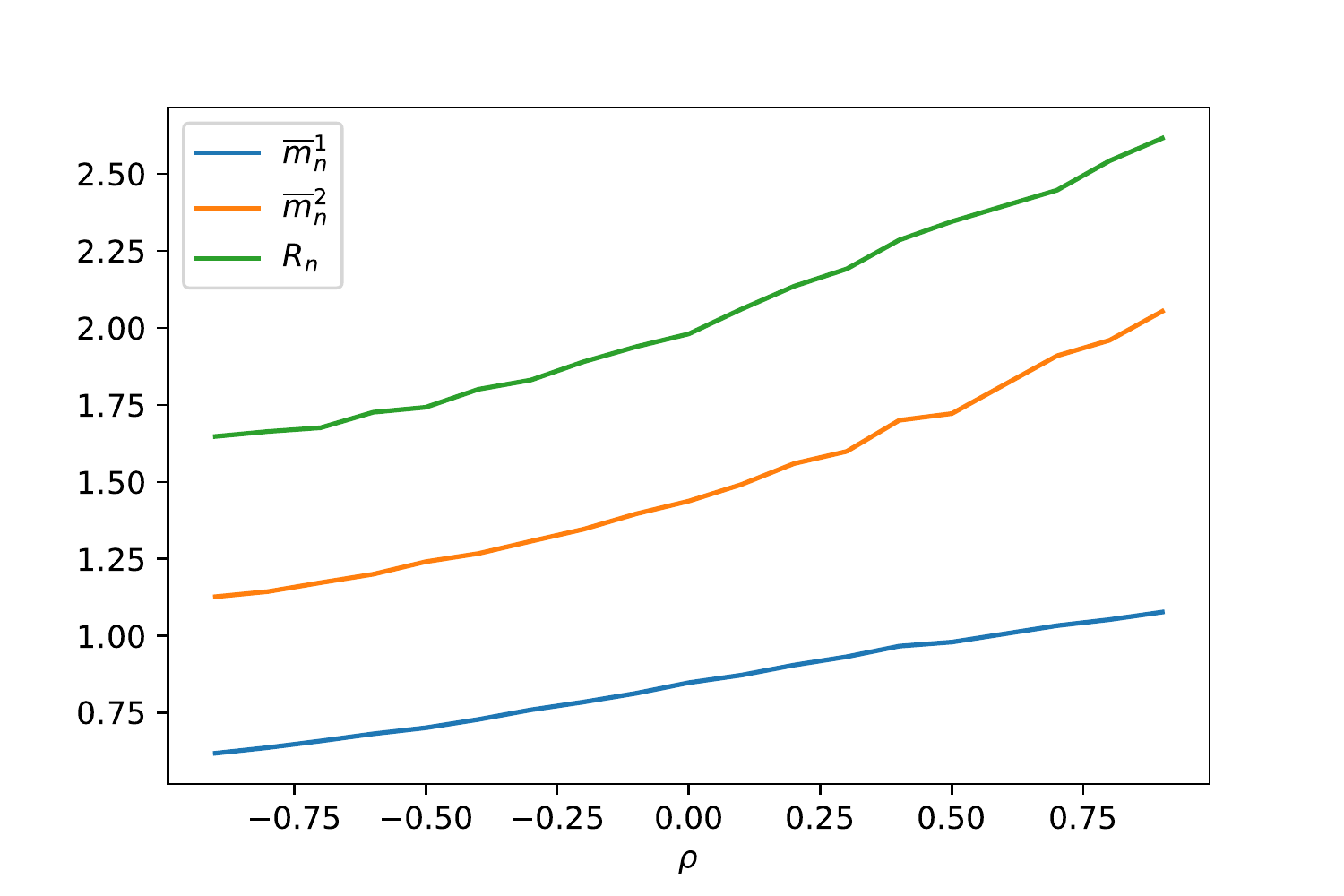}
    \caption{$R_n$, $\overline{m}_n^1$ and $\overline{m}_n^2$ as a function of $\rho$.}
    \label{loss1}
\end{figure}
\subsection{Second example}
\subsubsection{Simulated data}
In this example, we will be working with a Multivariate Normal Inverse Gaussian (MNIG) distribution for the vector $X$ instead of a Gaussian distribution. The MNIG distribution yields a more flexible family of distributions that can be skewed and have fatter tails than the Gaussian distribution. For a fixed $d$, a MNIG distributed random variable is a variance-mean mixture of a $d$-Gaussian random variable $Y$ with a univariate inverse Gaussian distributed mixing variable $Z$. The MNIG distribution has five parameters $\alpha_{MNIG} > 0, \beta \in \R^d, \delta > 0, \mu \in \R^d$ and $\Gamma \in \R^{d\times d}$ and can be constructed as follows:
\begin{equation}\label{MNIG}
X = \mu + Z\Gamma\beta + \sqrt{Z}\Gamma^{1/2} Y,
\end{equation} 
where $Z \sim IG(\delta^2, \alpha_{MNIG}^2-\beta^\intercal\Gamma\beta)$ and $IG(\chi,\psi)$ denotes the Inverse Gaussian distribution with parameters $\chi, \psi > 0$ and $Y \sim \mathcal{N}(0, I_d)$. Note that the random variable $X|Z \sim \mathcal{N}(\mu + Z\Gamma \beta, Z\Gamma)$, hence the name variance-mean mixture. The parameters of the MNIG distribution have natural interpretations. The parameter $\alpha_{MNIG}$ is a shape parameter and controls the heaviness of the tails. The parameter $\beta$ is a skewness vector parameter, $\delta$ is a scale parameter and $\mu$ is a vector translation parameter. Finally, the matrix $\Gamma$ is assumed to be a positive semidefinite symmetric matrix and controls the degree of correlations between components and assumed to be such that $\mathrm{det}(\Gamma)=1$. In order for the MNIG to exist, the inequality $\alpha_{MNIG}^2 > \beta^\intercal \Gamma \beta$ must be satisfied. The cumulant generating function of the MNIG could be derived easily in a closed form of the parameters:
$$\Phi_X(t) = \delta \left(\sqrt{\alpha_{MNIG}^2 - \beta^\intercal \Gamma \beta} - \sqrt{\alpha_{MNIG}^2 - (\beta + it)^\intercal\Gamma (\beta + i t)}\right) + i t^\intercal \mu.$$
This shows that the MNIG is infinitely divisible. Thus, we can easily evaluate the moments of this distribution. The mean vector and the covariance matrix $\Sigma$ of $X$ are given in the following:
\begin{align}
E[X] &= \mu + \frac{\delta \Gamma \beta}{\alpha_{MNIG}^2 - \beta^\intercal \Gamma \beta},\\
\Sigma &= \delta \left(\alpha_{MNIG}^2 - \beta^\intercal \Gamma \beta\right)^{-1/2}\left[\Gamma + \left(\alpha_{MNIG}^2 - \beta^\intercal \Gamma \beta\right)^{-1}\Gamma \beta \beta^\intercal \Gamma\right].
\end{align}
Note that due to the parameter $\beta$, even when $\mu = 0$ (and $\Gamma = I_d$ resp.), the mean of the MNIG distribution is not null (the covariance matrix is not diagonal resp.). For more details about MNIG, we refer to \citet{oigaard2004multivariate}.\\
In order to make the numerical analysis more realistic, we fitted, the parameters of the MNIG distribution on the daily log-return of three European indices: CAC 40, BEL 20 and AEX. The estimated parameters obtained using the Expectation Maximization (EM) algorithm, explained in details in Section \ref{estimationPar}, are summarized in the first column of the following table \ref{setParameters}.
\begingroup
\setlength{\tabcolsep}{6pt} % Default value: 6pt
\renewcommand{\arraystretch}{1}
\begin{table}[H]
\centering
\begin{tabular}{cc}
\hline
Parameters & MNIG\\
\hline
$\alpha_{MNIG}$ & $365.78$ \\
\hline
$\delta$ & $0.00373$ \\
\hline
$\beta$ & $(-64.28, 41.45, 7.35)$ \\
\hline
$\mu$ & $(0.00084, 0.00024, 0.00055)$ \\
\hline
$\Gamma$ & $\begin{pmatrix}
2.338 & 1.796 & 2.080\\
1.796 & 2.327 & 2.088\\
2.080 & 2.088 & 2.555
\end{pmatrix}$ \\
\hline
\end{tabular}
\caption{Parameters sets for the MNIG.}
\label{setParameters}
\end{table}
\endgroup
The covariance matrix obtained from the MNIG calibrated distribution is given in the following:
$$\Sigma = \begin{pmatrix}
2.45\times 10^{-5} & 1.86\times 10^{-5} & 2.16\times 10^{-5}\\
1.86\times 10^{-5} & 2.40\times 10^{-5} & 2.16\times 10^{-5}\\
2.16\times 10^{-5} & 2.16\times 10^{-5} & 2.65\times 10^{-5}
\end{pmatrix}.$$
This shows that the log-returns of the three indices over the period considered are barely correlated. The following figure shows also that they almost have the same distribution:
\begin{figure}[H]
    \centering
    \includegraphics[scale=0.9]{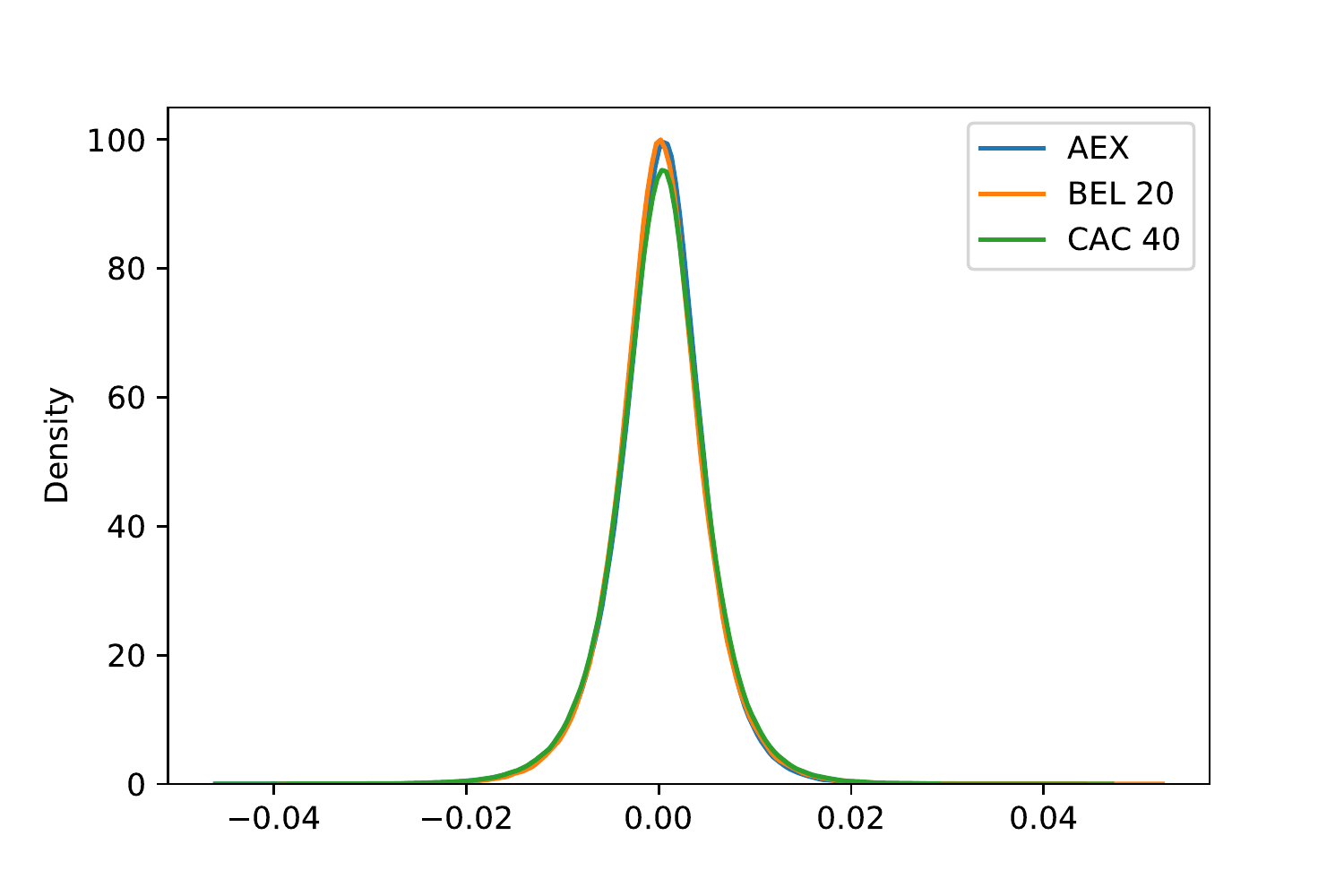}
    \caption{Densities of the log-returns of the three indices.}
    \label{densities}
\end{figure}
\subsubsection{Numerical Results}
We will test our SA algorithms with a trivariate MNIG distribution for the polynomial loss functions. We recall that the polynomial loss function is given by:
$$ l(x) = \sum_{i=1}^d \frac{([1 +x_i]^+)^{\theta_i} - 1}{\theta_i} + \alpha \sum_{i < j} \frac{([1 +x_i]^+)^{\theta_i}}{\theta_i} \frac{([1 +x_i]^+)^{\theta_j}}{\theta_j},~ \theta_i > 1,~\alpha \ge 0.$$
Since no closed formula is available to us in this case, we decided to use a Monte Carlo scheme as a benchmark to the SA method. This scheme  consists in approximating the expectation in \eqref{riskMeasure} by the corresponding Monte Carlo estimator and then to use Nelder-Mead algorithm as a minimization algorithm to find the optimal allocations. The compact set $K$ for the SA method was set to $[0, 2]^3$ and $\alpha$ was taken to be equal to $1$. First, we compare both methods in the case where the parameter $\theta$ was taken to be equal to $\theta = (2, 2, 2)$. Then, in a second case, we test both algorithms with the parameter $\theta = (1, 2, 3)$.
\begingroup
\setlength{\tabcolsep}{6pt} % Default value: 6pt
\renewcommand{\arraystretch}{1.2}
\begin{table}[H]
\centering
\begin{tabular}{cccc}
\hline
 & SA & CI-SA & Monte Carlo\\
\hline
$m_1^*$ & $0.31747$ & $[0.31746, 0.31749]$ & $0.31748$\\
\hline
$m_2^*$ & $0.31748$ & $[0.31746, 0.31750]$ & $0.31745$\\
\hline
$m_3^*$ & $0.31742$ & $[0.31740, 0.31743]$ & $0.31737$ \\
\hline
$R(X)$ & \multicolumn{2}{c}{$0.31336$} & $0.31332$ \\
\hline
CT(s) & \multicolumn{2}{c}{$141.20$} & $28.07$\\
\hline
\end{tabular}
\caption{Numerical results: Polynomial loss function with $\theta = (2, 2, 2)$ and MNIG distribution.}
\label{PolynomMNIG1}
\end{table}
\endgroup
The table \ref{PolynomMNIG1} show that both methods give approximately the same values for the optimal allocations $m^*$ as well as the risk measure $R(X)$. The values of the optimal allocations are approximately the same among the three components. This could be explained by the fact that the three components have almost the same distribution as shown in the figure \ref{densities} and the fact that we have taken $\theta = (2, 2, 2)$, so that the system becomes nearly symmetric. The Monte Carlo method is seven times faster that the SA method. However, with the Monte Carlo method, we do not have any confidence intervals and hence no control over the error of estimation. Moreover, since in the Monte Carlo method, we are using a deterministic minimization algorithm, it is sensitive to the initial values (Recall that we do not have this problem with the SA method). We do not have this problem with the SA method.
\begingroup
\setlength{\tabcolsep}{6pt} % Default value: 6pt
\renewcommand{\arraystretch}{1.2}
\begin{table}[H]
\centering
\begin{tabular}{cccc}
\hline
 & SA & CI-SA & Monte Carlo\\
\hline
$m_1^*$ & $0.21996$ & $[0.21995, 0.21998]$ & $0.21994$\\
\hline
$m_2^*$ & $0.25127$ & $[0.25125, 0.25129]$ & $0.25130$\\
\hline
$m_3^*$ & $0.29929$ & $[0.29927, 0.29931]$ & $0.29926$ \\
\hline
$R(X)$ & \multicolumn{2}{c}{$0.37532$} & $0.37529$ \\
\hline
CT(s) & \multicolumn{2}{c}{$98.48$} & $9.85$\\
\hline
\end{tabular}
\caption{Numerical results: Polynomial loss function with $\theta = (1, 2, 3)$ and MNIG distribution.}
\label{ExpMNIG1}
\end{table}
\endgroup
\section{Appendix: Estimation of MNIG parameters}
\label{estimationPar}
\subsection{Computational aspects}
In this section, we give more details about the estimation of the MNIG parameters. The most conventional way to estimate the latters is the maximum likelihood estimation method. However, in the case of MNIG, this method shows slow convergence due to the complexity of the likelihood. We therefore, propose here to use the Expectation Maximization (EM) algorithm which is known to be fast and accurate. The EM algorithm is a powerful tool that is used for maximum likelihood estimation for data containing \q{missing} values. This is suitable for distributions arising as mixtures which is the case of MNIG distributions where the mixing variable $Z$ is unobserved. The EM algorithm is an iterative algorithm that consists of two steps at each iteration. Denoting $\theta = (\delta, \mu, \beta, \alpha_{MNIG}, \Gamma)$, $X=(X_1,...,X_N)$ the observed data and $Z=(Z_1,...,Z_N)$ the unobserved one, $L(X, Z, \theta) = \log(\mathcal{P}(X, Z) | \theta)$ the complete data likelihood and $\theta^n$ the estimate of $\theta$ at step $n$, we repeat the two following steps until some convergence criteria is verified:
\begin{itemize}
\item E-step : Compute $Q(\theta | \theta^{n}) \vcentcolon= E_{Z | X, \theta^n}[L(X, Z, \theta^n)]$.
\item M-step : choose $\theta^{n+1} = \underset{\theta}{\argmax}~ Q(\theta | \theta^n)$.
\end{itemize}
Next, we explicit the calculations of $Q(\theta|\theta^n)$ in the E-step for the MNIG distribution. We have, by taking the constants away, 
\begin{align*}
L(X, Z, \theta) &= \log\left(\mathcal{P}(X, Z | \theta)\right)\\
&= \log\left(\mathcal{P}(X | Z, \theta)\right) + \log\left(\mathcal{P}(Z |\theta)\right)\\
&= -\frac{d}{2}\sum_{i=1}^{N} \log(Z_i) -\frac{N}{2}\log(\mathrm{det}(\Gamma))-\frac{1}{2}\sum_{i=1}^N\frac{1}{Z_i}(X_i - \mu - Z_i \Gamma \beta)^\intercal\Gamma^{-1}(X_i - \mu - Z_i \Gamma \beta)+\\
&N \sqrt{\delta^2(\alpha_{MNIG}^2 -\beta^\intercal\Gamma \beta)} + N\log(\delta)-\frac{3}{2}\sum_{i=1}^N \log(Z_i)-\frac{1}{2}\sum_{i=1}^N\left(\delta^2 \frac{1}{Z_i} + (\alpha_{MNIG}^2-\beta^\intercal \Gamma \beta) Z_i\right).
\end{align*}
Taking the conditional expectation on the both sides and denoting $\zeta^n = (\zeta_i^n)_{i=1,...,N}$ and $\phi^n = (\phi_i^n)_{i=1,...,N}$, where $\zeta_i^n \vcentcolon= E_{Z|X, \theta^n}[Z_i]$ and $\phi_i^n \vcentcolon= E_{Z|X, \theta^n}[\frac{1}{Z_i}]$, we get, again by removing the quantities that does not depend on $\theta$,
\begin{align*}
Q(\theta |\theta^n) &=- \frac{N}{2}\log(\mathrm{det}(\Gamma))-\frac{1}{2}\sum_{i=1}^N \left(\phi_i^n(X_i-\mu)^\intercal \Gamma^{-1}(X_i-\mu) + \zeta_i^n\beta^\intercal \Gamma \beta - 2(X_i -\mu)^\intercal \beta\right)\\
&+ N\sqrt{\delta^2(\alpha_{MNIG}^2 -\beta^\intercal\Gamma \beta)} + N\log(\delta)-\frac{1}{2}\sum_{i=1}^N\left(\phi_i^n \delta^2 + \zeta_i^n(\alpha_{MNIG}^2-\beta^\intercal \Gamma \beta) \right).
\end{align*}
The quantities $\phi_i^n$ and $\zeta_i^n$ can be derived from the fact that $Z|X,\theta$ follows a Generalized Inverse Gaussian distribution, i.e., $Z|X,\theta \sim GIG\left(-\frac{d+1}{2}, q(X), \alpha_{MNIG}\right)$, where $q$ is given as:
\begin{equation}
q(x) = \sqrt{\delta^2 + (x-\mu)^\intercal \Gamma^{-1} (x-\mu)}.
\end{equation}
More precisely, they are given by,
\begin{align}\label{zeta}
\zeta_i &\vcentcolon= E_{Z_i|X_i,\theta}[Z_i] = \frac{q(X_i)}{\alpha}\frac{K_{(d-1)/2}(\alpha q(X_i))}{K_{(d+1)/2}(\alpha q(X_i))},\\ \label{phi}
\phi_i &\vcentcolon= E_{Z_i|X_i,\theta}[1 / Z_i] = \frac{\alpha}{q(X_i)}\frac{K_{(d+1)/2}(\alpha q(X_i))}{K_{(d+3)/2}(\alpha q(X_i))},
\end{align}
$K_v$ is the modified Bessel function of the second kind with
index $v\in \R$.
Having calculated $Q(\theta|\theta^n)$, we now need to calculate the next term $\theta^{n+1} \vcentcolon= \underset{\theta}{\argmax} Q(\theta|\theta^n)$. This will be done by first calculating the gradient of $Q$. 
$$
\left\{
\begin{aligned}
\frac{\partial Q}{\partial \delta} &= N\sqrt{\alpha_{MNIG}^2 -\beta^\intercal\Gamma \beta} + \frac{N}{\delta} - \delta \sum_{i=1}^N \phi_i^n,\\
\frac{\partial Q}{\partial \alpha_{MNIG}} &=  \frac{N \delta \alpha}{\sqrt{\alpha_{MNIG}^2 -\beta^\intercal\Gamma \beta}} - \alpha \sum_{i=1}^N \zeta_i^n,\\
\frac{\partial Q}{\partial \mu} &= \Gamma^{-1}\sum_{i=1}^N \phi_i^n (X_i - \mu) - N\beta,\\
\frac{\partial Q}{\partial \beta} &= \sum_{i=1}^N X_i -N \mu - \frac{N\delta\Gamma \beta}{\sqrt{\alpha_{MNIG}^2 -\beta^\intercal\Gamma \beta}},\\
\frac{\partial Q}{\partial \Gamma} &= \frac{1}{2}\left(\Gamma^{-1}\sum_{i=1}^N(\phi_i^n (X_i -\mu)(X_i-\mu)^\intercal \Gamma^{-1}-N\Gamma^{-1}-\frac{N\delta \beta \beta^\intercal}{\sqrt{\alpha_{MNIG}^2 -\beta^\intercal\Gamma \beta}}\right).
\end{aligned}\right.
$$
To alleviate the expressions, we will denote $\overline{\phi}^n \vcentcolon = \frac{1}{N}\sum_{i=1}^N \phi_i^n$,
$\overline{\zeta}^n \vcentcolon = \frac{1}{N}\sum_{i=1}^N \zeta_i^n$, $\overline{X\phi}^n \vcentcolon = \frac{1}{N}\sum_{i=1}^N \phi_i^n X_i$ and $\overline{X} \vcentcolon =\frac{1}{N}\sum_{i=1}^N X_i$ . Setting the previous set of equations to $0$, we obtain,
\begin{align}
&\frac{1}{\delta} - \overline{\phi}^n\delta + \sqrt{\alpha_{MNIG}^2 -\beta^\intercal\Gamma \beta} = 0,\\
&\frac{\delta}{\sqrt{\alpha_{MNIG}^2 -\beta^\intercal\Gamma \beta}} - \overline{\zeta}^n = 0,\\
&\Gamma^{-1}\overline{X\phi}^n-\overline{\phi}^n\Gamma^{-1}\mu - \beta = 0,\\
&\overline{X}-\mu-\frac{\delta \Gamma \beta}{\alpha_{MNIG}^2 -\beta^\intercal\Gamma \beta} = 0,\\
&\left(\frac{1}{N}\sum_{i=1}^N\phi_i(X_i -\mu)(X_i - \mu)^\intercal\right) - \Gamma -\frac{\delta\Gamma\beta \beta^\intercal\Gamma}{\alpha_{MNIG}^2 -\beta^\intercal\Gamma \beta} = 0.
\end{align}
From the second equation we deduce that,
\begin{equation}\label{frac}
\frac{\delta}{\sqrt{\alpha_{MNIG}^2 -\beta^\intercal\Gamma \beta}} = \overline{\zeta}^n.
\end{equation}
Plugging this into the first equation gives us,
\begin{equation}\label{delta}
\delta = \frac{1}{\sqrt{\overline{\phi}^n-\frac{1}{\overline{\zeta}^n}}}.
\end{equation}
Thanks to the third equation, we have,
\begin{equation}\label{gammaBeta}
\Gamma \beta = \overline{X\phi}^n-\mu \overline{\phi}^n.
\end{equation}
Using this in the fourth equation, we obtain,
\begin{equation}\label{mu}
\mu = \frac{\overline{X}-\overline{\zeta}^n~\overline{X\phi}^n}{1 -\overline{\zeta}^n~\overline{\phi}^n}.
\end{equation}
Now that $\mu$ and $\Gamma \beta$ are explicitly known, denoting $R \vcentcolon = \frac{1}{N}\sum_{i=1}^N\phi_i(X_i -\mu)(X_i - \mu)^\intercal - \overline{\zeta}^n\Gamma\beta (\Gamma \beta)^\intercal$, we get from the fifth equation,
\begin{equation}\label{Gamma}
\Gamma = \frac{R}{\mathrm{det}(R)^{1/d}}.
\end{equation}
Going back to \eqref{gammaBeta}, we get,
\begin{equation}\label{beta}
\beta = \Gamma^{-1}(\overline{X\phi}^n- \overline{\phi}^n\mu).
\end{equation}
Finally, using \eqref{frac}, $\alpha_{MNIG}$ can be deduced as:
\begin{equation}\label{alpha}
\alpha = \sqrt{\left(\frac{\delta^2}{\overline{\zeta}^n}\right) + \beta^\intercal\Gamma\beta}.
\end{equation} 
In the following, we summarize the EM algorithm for the parameters estimation of MNIG distribution:
\begin{algorithm}
\caption{EM algorithm for parameters estimation of MNIG distribution}\label{algo}
\KwInput{Observations $(X_1,..., X_N)$, initial value~$\theta^0$, tolerance $\mathrm{tol}$ and $M$ number of iterations\;}
Set $\theta^n = \theta^0$ and compute $\phi^n$, $\zeta^n$ with \eqref{zeta} and \eqref{phi}\;
Compute $\theta^{n+1}$ using in order \eqref{delta}, \eqref{mu}, \eqref{gammaBeta}, \eqref{Gamma}, \eqref{beta} and  \eqref{alpha}\;
$n = 0$\;
\While{$||\theta^{n+1} - \theta^n|| \ge \mathrm{tol}$ and $n < M$}{
E-step: $\theta^n \leftarrow \theta^{n+1}$ and compute the new $\phi^n$ and $\zeta^n$ with \eqref{zeta} and \eqref{phi}\;
M-step: Compute $\theta^{n+1}$ using in order \eqref{delta}, \eqref{mu}, \eqref{gammaBeta}, \eqref{Gamma}, \eqref{beta} and  \eqref{alpha}\;
$n \leftarrow n+1$\;}
\KwOutput{Estimated parameters $\hat{\theta}$\;}
\end{algorithm}
The convergence properties of the EM algorithm are discussed in details in \citet{mclachlan2007algorithm}. However, to avoid getting stuck in a local maximum, we will need to run the algorithm from several starting points to ensure that the obtained maximum is the global one. We can also combine the algorithm with other numerical methods, such as Newton-Raphson algorithm, to speed up the convergence.
\subsection{Numerical aspects}
We applied the EM algorithm described in the above subsection to a data set of daily log return of three European stock indices (CAC 40, BEL 20 and AEX) for a period from $12/05/2020$ to  $10/05/2022$ obtained from the website of Euronext. The data set consisted of $514$ observations. In order to test the behavior of the algorithm, several initial values were considered. Note that the conditional expectations in the E-step do not involve the parameters $\beta$ and hence the convergence of the algorithm will not depend on the initial value of $\beta$. We fixed $\beta = (0, 0, 0)$ and we stopped the iterations when $||\theta^{n+1}-\theta^{n}|| < \mathrm{tol}$ for $\mathrm{tol} = 10^{-5}$ and $\mathrm{tol} = 10^{-10}$. The values of the estimates of estimates for initial values were the same and are given in the following:
\begin{align*}
\hat{\alpha} &= 365.78,\\
\hat{\delta} &= 0.00373,\\
\hat{\beta} &= (-64.28, 41.45, 7.35),\\
\hat{\mu} &= (0.00084, 0.00024, 0.00055),\\
\hat{\Gamma} &= \begin{pmatrix}
2.338 & 1.796 & 2.080\\
1.796 & 2.327 & 2.088\\
2.080 & 2.088 & 2.555
\end{pmatrix}.
\end{align*}
The number of iterations needed until convergence along with the computational (CT) time can be seen in \ref{fitting}. 
\begingroup
\setlength{\tabcolsep}{6pt} % Default value: 6pt
\renewcommand{\arraystretch}{1.2}
\begin{table}[H]
\centering
\begin{tabular}{cccccccc}
\hline
\multicolumn{4}{l}{Initial Values} & \multicolumn{2}{c}{$\mathrm{tol} = 10^{-5}$}& \multicolumn{2}{c}{$\mathrm{tol} = 10^{-10}$}\\
\hline
$\alpha$ & $\delta$ & $\mu$ & $\Gamma$ & Iterations & CT($\mathrm{ms}$) & Iterations & CT($\mathrm{ms})$\\
\hline
$1$ & $0$ & $(0,0,0)$ & $I_3$ & $93$ & $1267$ & $164$ & $2175$\\
$20$ & $0$ & $(0,0,0)$ & $I_3$ & $93$ & $1174$ & $164$ & $2139$\\
$1$ & $2$ & $(0,0,0)$ & $I_3$ & $194$ & $2428$ & $265$ & $3377$\\
$20$ & $2$ & $(0,0,0)$ & $I_3$ & $453$ & $5543$ & $524$ & $6471$\\
$1$ & $0$ & $(1,1,1)$ & $I_3$ & $189$ & $2425$ & $260$ & $3391$\\
$20$ & $0$ & $(1,1,1)$ & $I_3$ & $410$ & $5058$ & $481$ & $5898$ \\
$1$ & $2$ & $(1,1,1)$ & $I_3$ & $206$ & $2563$ & $217$ & $3393$\\
$20$ & $2$ & $(1,1,1)$ & $I_3$ & $558$ & $6658$ & $629$ & $7644$\\
$1$ & $0$ & $(0,0,0)$ & $2I_3$ & $94$ & $1237$ & $165$ & $2078$\\
$20$ & $0$ & $(0,0,0)$ & $2I_3$ & $94$ & $1154$ & $165$ & $2100$\\
$1$ & $2$ & $(0,0,0)$ & $2I_3$ & $194$ & $2383$ & $265$ & $3315$\\
$20$ & $2$ & $(0,0,0)$ & $2I_3$ & $453$ & $5536$ & $524$ & $6558$\\
$1$ & $0$ & $(1,1,1)$ & $2I_3$ & $177$ & $2160$ & $248$ & $3032$\\
$20$ & $0$ & $(1,1,1)$ & $2I_3$ & $328$ & $3947$ & $399$ & $4876$ \\
$1$ & $2$ & $(1,1,1)$ & $2I_3$ & $201$ & $2468$ & $272$ & $3282$\\
$20$ & $2$ & $(1,1,1)$ & $2I_3$ & $509$ & $6208$ & $580$ & $7118$\\
\hline
\end{tabular}
\caption{Number of iterations and computational time for various initial values and stopping criteria.}
\label{fitting}
\end{table}
\endgroup

\bibliographystyle{apa}
\bibliography{biblio}  %%% Uncomment this line and comment out the ``thebibliography'' section below to use the external .bib file (using bibtex) .





% \end{thebibliography}

\end{document}